\documentclass[12pt,leqno]{article}
\usepackage{amsmath,amssymb,amscd,latexsym,bm}

\newcommand{\dis}{\displaystyle}
\usepackage[all]{xy}
\newcommand{\A}{{\mathbb{A}}}

\newcommand{\C}{{\mathbb{C}}}

\newcommand{\F}{{\mathbb{F}}}

\newcommand{\N}{\mathbb{N}}

\newcommand{\Q}{{\mathbb{Q}}}
\newcommand{\oQ}{\overline{\Q}}

\newcommand{\Z}{{\mathbb{Z}}}
\newcommand{\oZ}{\overline{\Z}}
\newcommand{\uZ}{\underline{\Z}}

\newcommand{\tf}{\tilde{f}}
\newcommand{\tK}{\tilde{K}}
\newcommand{\tQ}{\tilde{Q}}

\newcommand{\car}{\mathrm{char}}
\newcommand{\cont}{\mathrm{cont}}
\newcommand{\Corr}{\mathrm{Corr}}
\newcommand{\uCorr}{\underline{\Corr}}
\newcommand{\cycle}{\mathrm{cycle}}
\newcommand{\length}{\mathrm{length}}
\newcommand{\equi}{\mathrm{equi}}

\newcommand{\id}{\mathrm{id}}

\renewcommand{\mod}{\mathrm{mod}\,}
\newcommand{\Mor}{\mathrm{Mor}}
\newcommand{\NAS}{\mathrm{NoethAffSch}}
\newcommand{\AS}{\mathrm{AffSch}}

\newcommand{\op}{\mathrm{op}}
\newcommand{\red}{\mathrm{red}}
\newcommand{\rk}{\mathrm{rk}\,}

\newcommand{\spec}{\mathrm{spec}\,}

\newcommand{\supp}{\mathrm{supp}\,}
\newcommand{\Aut}{\mathrm{Aut}}

\newcommand{\End}{\mathrm{End}}
\newcommand{\Equ}{\mathrm{Equ}}
\newcommand{\Ext}{\mathrm{Ext}}
\newcommand{\uExt}{\underline{\Ext}}

\newcommand{\Gal}{\mathrm{Gal}}

\newcommand{\Hom}{\mathrm{Hom}}

\newcommand{\ind}{\mathrm{ind}}

\newcommand{\Ker}{\mathrm{Ker}\,}

\newcommand{\PropCycl}{\mathrm{PropCycl}}
\newcommand{\PropCaDir}{\mathrm{PropCaDiv}}

\newcommand{\rat}{\mathrm{rat}}

\newcommand{\sep}{\mathrm{sep}}

\newcommand{\Ch}{{\mathcal C}}

\newcommand{\Mh}{{\mathcal M}}
\newcommand{\Nh}{{\mathcal N}}
\newcommand{\Oh}{{\mathcal O}}
\newcommand{\Ph}{\mathcal{P}}
\newcommand{\eg}{{\mathfrak{g}}}

\newcommand{\ep}{\mathfrak{p}}

\newcommand{\onA}{\overline{A}}

\newcommand{\oK}{\overline{K}}

\newcommand{\uR}{\underline{R}}

\newcommand{\ob}{\overline{b}}
\newcommand{\oX}{\overline{X}}
\newcommand{\ox}{\overline{x}}

\newcommand{\tA}{\tilde{A}}

\newcommand{\tP}{\tilde{P}}

\newcommand{\dcup}{\dot{\cup}}
\newcommand{\silo}{\stackrel{\sim}{\rightarrow}}
\newcommand{\ent}{\,\hat{=}\,}
\newcommand{\colim}{\mathrm{colim}}
\newcommand{\Affsch}{\mathrm{AffSch}}
\newcommand{\Indsch}{\mathrm{IndSch}}
\newcommand{\Sch}{\mathrm{Sch}}
\newcommand{\Ind}{\mathrm{Ind}}
\newcommand{\rdp}{\!^{\dis\prime}\;}
\newcommand{\btu}{\bigtriangleup}
\newtheorem{theorem}{Theorem}[section]
\newtheorem{lemma}[theorem]{Lemma}
\newtheorem{prop}[theorem]{Proposition}
\newtheorem{defn}[theorem]{Definition}
\newtheorem{cor}[theorem]{Corollary}
\newtheorem{example}[theorem]{Example}
\newtheorem{examples}[theorem]{Examples}
\newtheorem{remark}[theorem]{Remark}

\newcounter{aufz}
\newenvironment{rem}{\noindent {\bf Remark}}{}
\newenvironment{rems}{\noindent {\bf Remarks}}{}

\newenvironment{exmp}{\noindent{\bf Example}}{}

\newenvironment{proofof}{\noindent {\bf Proof of}}{\mbox{}\hspace*{\fill}$\Box$}
\newenvironment{proof}{\noindent {\bf Proof}}{\mbox{}\hspace*{\fill}$\Box$}

\newcommand{\pr}{\mathrm{pr}}

\newcommand{\tX}{\tilde{X}}

\newcommand{\Quot}{\mathrm{Quot}}

\newcommand{\verk}{\mbox{\scriptsize $\,\circ\,$}}

\begin{document}
\title{Rational Witt vectors and associated sheaves}
\author{Christopher Deninger\footnote{Funded by the Deutsche Forschungsgemeinschaft (DFG, German Research Foundation) under Germany's Excellence Strategy EXC 2044--390685587, Mathematics M\"unster: Dynamics--Geometry--Structure and the CRC 1442 Geometry: Deformations and Rigidity}}
\date{}
\maketitle



\section*{Introduction}
The rational functions within the big Witt vector ring $W (A) = 1 + TA [[T]]$ form a subring $W_{\rat} (A)$ for any commutative unital ring $A$. It carries Frobenius and Verschiebung endomorphisms and may be viewed as an uncompleted version of $W (A)$. In \cite{A} Almkvist gave a very natural $K$-theoretic interpretation of $W_{\rat} (A)$ see also \cite{Gr} which was extended to higher $K$-theory in \cite{St1}, \cite{St2}. Basic works on rational Witt vectors include \cite{A, DMP, BV, CC, C1, H, Sch}.

For reasons explained below and related to \cite{KS} and \cite{D} we study the presheaf $W_{\rat} (\Oh)$ defined by the rational Witt vectors in various Grothendieck topologies both subcanonical and non subcanonical. For this purpose, works on the Fatou property of rings and the rationality properties of power series turn out to be relevant, \cite{B, CC, C1, H}. For example it follows that on a normal, Noetherian scheme the presheaf $W_{\rat} (\Oh) = W_{\rat} \verk \Oh$ is already a sheaf in the Zariski topology (Theorem \ref{t3.3}). Moreover by \cite{Sch}, the presheaf $W_{\rat} (\Oh)$ satisfies the $fpqc$-scheaf condition on the category of Noetherian affine schemes (Theorem \ref{t3.4}). For a ring $A$ let $\Z A$ be the monoid ring on the multiplicative monoid $(A , \cdot)$ and let $\uZ A = \Z A / \Z (0)$ be the reduced monoid ring obtained by dividing by the ideal $\Z (0)$. For a ring $R$, ring-homomorphisms $\uZ A \to R$ correspond to multiplicative maps $A \to R$ with $1 \mapsto 1$ and $0 \mapsto 0$. The natural ring homomorphisms $\uZ A \to W_{\rat} (A)$ give a map of presheaves $\uZ \Oh \to W_{\rat} (\Oh)$ and we give criteria when this map becomes surjective (Proposition \ref{t4.3}) resp. injective (Proposition \ref{t4.5}) after passing to associated sheaves. Surjectivity can be achieved in a subcanonical topology but injectivity not (Example \ref{t4.4}). However injectivity can be easily achieved in non-subcanonical topologies. The question to what extent $W_{\rat} (\Oh)$ is a sheaf for subcanonical topologies also relates to the $\ind$-scheme $W_J$ introduced by Hazewinkel using Hankel determinants in \cite{H}. On Fatou, e.g. Noetherian domains $W_J$ represents $W_{\rat}$ but not in general. We prove that $W_J$ is an $\ind$-{\it ring} scheme with Frobenius and Verschiebung endomorphisms (Theorem \ref{t2.8}). This does not follow from formal reasons since $W_J$ is not an $\ind$-Fatou scheme. In Theorem \ref{t3.4} and Proposition \ref{t4.8} we relate $W_{\rat} (\Oh)$ and $W_J (\Oh)$ in different topologies. 

The presheaf of groups $\Z \Oh$ and its $h$- and $qfh$-sheafification appear as special cases of a more general construction in Voevodsky's work \cite{V}. Over normal Noetherian bases the sections of the associated sheaves can be expressed in terms of groups of finite cycles \cite[\S\,6]{SV1}. On the other hand they are related to rational Witt vectors by the isomorphisms \eqref{eq:21}. In fact it is not difficult to directly prove an isomorphism of these rational Witt vector rings with rings of finite algebraic cycles (Theorem \ref{t5.1}). We also sketch a possible generalization of Almkvist's theorem motivated by this observation. 

In \cite{KS} and \cite{D} using rational Witt vectors, topological spaces resp. topological dynamical systems were constructed that are related to Galois theory and arithmetic geometry. See also \cite{L}. These new spaces have some useful properties but also certain shortcomings as explained in the respective introductions. The motivation for the present paper was to understand better why rings of rational Witt vectors arise at all in the constructions of \cite{KS} and \cite{D} and then to possibly improve on them. One answer why $W_{\rat}$ appears is the following argument, which is inspired by the old ``numbers are functions'' idea. Assume that for rings $A$ we have natural spaces $X_A$ on which ``numbers'' i.e. the elements $a$ of $A$ become non-trivial complex valued functions $f_a$. The map $a \mapsto f_a$ should be multiplicative with $0 \mapsto 0$ and $1 \mapsto 1$, but in general it cannot be additive since otherwise e.g. for $A = \Z$ the functions $f_a$ would be constant. The multiplicative map $A \to C (X_A , \C) , a \mapsto f_a$ into the ring of continuous complex valued functions on $X_A$ induces a ring homomorphism
\begin{equation}
\label{eq:01}
\uZ A \longrightarrow C (X_A , \C) \; .
\end{equation}
Rational Witt vectors appear as follows for normal domains $A$ with perfect quotient field $K$ if the numbers $\ent$ functions correspondence satisfies the following descent property. Let $L$ be a finite Galois extension of $K$ with Galois group $N$ and let $B$ be the normalization of $A$ in $L$. Assume that the natural map $X_B \to X_A$ induces an isomorphism $X_B / N \silo X_A$. Then the map \eqref{eq:01} extends to a map
\begin{equation}
\label{eq:02}
(\uZ B)^N \longrightarrow C (X_B , \C)^N = C (X_B / N , \C) = C (X_A , \C) \; .
\end{equation}
Let $\overline{A}$ be the normalization of $A$ in an algebraic closure $\oK$ of $K$ and let $G = \Gal (\oK / K)$. In the colimit we obtain from \eqref{eq:02} a ring homomorphism, c.f. Theorem \ref{t1.1},
\begin{equation}
\label{eq:03}
W_{\rat} (A) = (\uZ \onA)^G \longrightarrow C (X_A , \C) \; .
\end{equation}
Hence rational Witt vectors appear naturally as functions on $X_A$. Conversely, rings of rational Witt vectors were used in \cite{KS} and \cite{D} to construct suitable spaces $X_A$. 

The algebra $\uZ A$ knows nothing about the addition in $A$. However, the larger $G$ is, the more information the ring $W_{\rat} (A)$ has about the additive structure of $A$. It may be interesting to experiment with stronger descent conditions than the Galois descent above to obtain replacements of $W_{\rat} (A)$ which know even more about the additive structure of $A$ to remedy the defects in the constructions of \cite{KS} and \cite{D}.

A part of this work was done at Imperial College in London and at the Vietnam Institute of Mathematics in Hanoi. I would like to thank these institutions and my hosts Kevin Buzzard and Ho Hai Phung very much.
\section{Background on rational Witt vectors} \label{sec:1}

For a commutative unital ring $A$ let $A [T]_S$ be the localization of the polynomial ring $A [T]$ at the multiplicative subset $S = \{ f \in A [T] \mid f (0) = 1 \}$. As an abelian group, $W_{\rat} (A)$ is the kernel of the evaluation homomorphism $A [T]^{\times}_S \to A^{\times} , f \mapsto f (0)$. The embeddings $A [T] \subset A [[T]]$ and hence $A [T]_S \subset A [[T]]$ show that $W_{\rat} (A)$ is a subgroup of $W (A) = 1 + T A [[T]] \subset A [[T]]^{\times}$. It follows that the functor $W_{\rat}$ respects monomorphisms. By definition $W_{\rat}$ respects epimorphisms and filtered colimits. It turns out that $W_{\rat} (A)$ is actually a subring of the big Witt ring $W (A)$. Moreover the Frobenius and Verschiebung maps $F_N$ and $V_N$ of $W (A)$ leave $W_{\rat} (A)$ invariant for all $N \ge 1$. The multiplicative, unital and injective map $[\,] : A \to W (A) , [a] = 1-at$ factors over $W_{\rat} (A)$ and induces a functorial ring homomorphism
\begin{equation}
\label{eq:1}
\omega : \uZ A \longrightarrow W_{\rat} (A) \; .
\end{equation}
Here $\uZ A = \Z A / \Z (0)$ is the reduced monoid algebra of $(A , \cdot)$. The idempotent $e = (0)$ in $\Z A$ gives rise to an isomorphism of rings 
\[
\Z A = e \Z A \times (1-e) \Z A = \Z (0) \times (\Z A)^0 \; .
\]
Here the ideal $(\Z A)^0 = \Ker (\deg : \Z A \to \Z)$ of $\Z A$ is viewed as a ring with unit $1 -e = 1 - (0)$. We have an isomorphism of unital rings $(\Z A)^0 \silo \uZ A$ via $x \mapsto \ox$ with inverse map $\ox \mapsto x - (\deg x) (0)$. If $A$ is a domain we can also identify $\uZ A$ with $\Z (A \setminus 0)$, the monoid algebra of $(A \setminus 0 , \cdot)$. 

The functor of rational Witt vectors commutes with localization in the following sense. Let $M$ be a multiplicatively closed subset of $A$ and let $A_M = M^{-1} A$ be the localization. Then $[M] = \{ [m] \mid m \in M \}$ is a multiplicatively closed subset of $W_{\rat} (A)$. The natural map $W_{\rat} (A) \to W_{\rat} (A_M)$ sends the elements of $[M]$ to units. Hence we get an induced map
\[
W_{\rat} (A)_{[M]} = [M]^{-1} W_{\rat} (A) \longrightarrow W_{\rat} (A_M) \; .
\]

\begin{prop}
\label{t1.1n}
The ring homomorphism
\[
W_{\rat} (A)_{[M]} \longrightarrow W_{\rat} (A_M)
\]
is an isomorphism.
\end{prop}

\begin{proof}
Let $\odot$ denote the multiplication in $W (B)$, where $B$ is a commutative unital ring. For $b \in B$ and $f = f (T)$ in $W (B)$ we have the formula
\[
[b] \odot f = (1 - bT) \odot f (T) = f (bT) \quad \text{in} \; W (B) \; .
\]
Hence the same formula holds in $W_{\rat} (B)$ if $f \in W_{\rat} (B)$. For $\tf = \tP / \tQ$ in $W_{\rat} (A_M)$ with $\tP , \tQ \in A_M [T] , \tP (0) = 1 = \tQ (0)$ there exist $m \in M$ and $a_1 , \ldots , a_r , b_1 , \ldots , b_r \in A$ with
\[
\tP (T) = 1 + \frac{a_1}{m} T + \ldots + \frac{a_r}{m^r} T^r \quad \text{and} \quad \tQ (T) = 1 + \frac{b_1}{m} T + \ldots + \frac{b_r}{m^r} T^r \; .
\]
Consider $P (T) = 1 + a_1 T + \ldots + a_r T^r$ and $Q (T) = 1 + b_1 T + \dots + b_r T^r$ in $A [T]$. Then $f = P / Q$ in $W_{\rat} (A)$ maps to $[m / 1] \odot \tf$ in $W_{\rat} (A_M)$ where $m / 1$ denotes the image of $m$ in $A_M$. Hence the quotient of $f$ by $[m]$ in $W_{\rat} (A)_{[M]}$ maps to $\tf$ in $W_{\rat} (A_M)$. This shows surjectivity. For $f = P / Q$ in $W_{\rat} (A)$ and $m \in M$, assume that the quotient of $f$ by $[m]$ in $W_{\rat} (A)_{[M]}$ maps to zero in $W_{\rat} (A_M)$. Using the same notation as above this means that $\tP (T) = \tQ (T)$ in $A_M [T]$ or equivalently that $a_i / m^i = b_i / m^i$ for $1 \le i \le r$. Hence there is some element $m_1 \in M$ with $m_1 a_i = m_1 b_i$ and therefore also $m^i_1 a_i = m^i_1 b_i$ i.e. $P (m_1 T) = Q (m_1 T)$. Thus $[m_1] \odot P = [m_1] \odot Q$ and hence $[m_1] \odot f$ is the constant polynomial $1$ i.e. the zero element of $W_{\rat} (A)$. Hence the image of $f$ in $W_{\rat} (A)_{[M]}$ is zero and injectivity follows. 
\end{proof}

A domain is called absolutely integrally closed (aic) if it is integrally closed with algebraically closed quotient field. 

The rational Witt vector ring of an integral domain $A$ with quotient field $K$ has the following description. Fix an algebraic closure $\oK$ of $K$ and let $G = \Aut_{K} (\oK)$ be the automorphism group of $\oK$ over $K$. Then $G$ acts on the integral closure $\onA$ of $A$ in $\oK$. For an element $0 \neq r \in \oK$ let $d_r$ be the degree of inseparability of $r$ over $K$ if $\car K = p$ and $d_r = 1$ if $\car K = 0$. Let $\uZ \onA / A$ be the subring (!) of $\uZ \onA$ consisting of elements $\sum_{a \in \onA} n_a (a) \mod \Z (0)$ where $d_a$ divides $n_a$ for all $a \neq 0$. Since $\onA$ is aic, the map $\omega$ in \eqref{eq:1} is an isomorphism, $W_{\rat} (\onA) = \uZ \onA$, c.f. \cite{D}, Proposition 1.1 a). The following known result describes $W_{\rat} (A) \subset W_{\rat} (\onA)$ as a subring of $\uZ \onA$. 

\begin{theorem}
\label{t1.1}
For an integral domain, with notations as above, we have
\begin{equation}
\label{eq:2}
W_{\rat} (A) = (\uZ \onA / A)^G \; .
\end{equation}
Tensoring with $\Z [1 / p]$ if $p = \car K > 0$ gives an isomorphism
\begin{equation}
\label{eq:3}
W_{\rat} (A) \otimes \Z [1 / p] = (\uZ \onA)^G \otimes \Z [ 1/p ] \; .
\end{equation}
If $K$ is perfect, we have
\begin{equation}
\label{eq:4}
W_{\rat} (A) = (\uZ \onA)^G \; .
\end{equation}
\end{theorem}

\begin{proof}
By \cite{BV} \S\,2, see also \cite{D}, Proposition 1.4, we have
\[
W_{\rat} (K) = (\uZ \oK / K)^G \; .
\]
From \cite{CC}, Proposition 2.5 we get
\[
W_{\rat} (A) = W_{\rat} (\onA) \cap W_{\rat} (K) \; .
\]
This implies \eqref{eq:2}
\[
W_{\rat} (A) = \uZ \onA \cap (\uZ \oK / K)^G = (\uZ \onA / A)^G \; .
\]
Assertions \eqref{eq:3} and \eqref{eq:4} are immediate consequences of \eqref{eq:2}
\end{proof}

\begin{remark}
\label{t1.2}
\em If $A$ is a ring with an element $a \neq 0$ with $a^2 = 0$ then the non-zero element $2 (a) - (2a) \in \uZ A$ is in the kernel of the map $\omega$. If $pA = 0$ for some prime $p$ and some $a \neq 0$ satisfies $a^p = 0$, then $\omega$ and the $p$-multiplication map on $W_{\rat} (A)$ are not injective. Namely $\omega ((a)) \neq 0$ in $W_{\rat} (A)$ and $p (a) \neq 0$ in $\uZ A$, whereas $\omega (p (a)) = p \omega ((a)) = 0$ in $W_{\rat} (A)$ because
\[
(1 - at)^p = 1 - a^p t^p = 1 \ent 0 \quad \text{in} \; W_{\rat} (A) \; .
\]
\end{remark}

Let $A$ be an integral domain with quotient field $K$. An element $x \in K$ is integral over $A$ if $A [x]$ is a finite $A$-module. It is called quasi-integral over $A$ if $A [x]$ is contained in a finitely generated $A$-submodule of $K$ or equivalently, if there is some $0 \neq d \in A$ such that $dA [x] \subset A$ i.e. $dx^n \in A$ for all $n \ge 1$, c.f. \cite{K}.

The next result is essentially due to Cahen and Chabert:

\begin{theorem}
\label{t1.3}
The following conditions on an integral domain $A$ with quotient field $K$ are equivalent:\\
a) Every element $x \in K$ which is quasi-integral over $A$ is integral over $A$.\\
b) $A$ is a Fatou domain i.e. an integral domain such that in $K ((T)) = \Quot (K [[T]])$ we have the equality
\[
A [T]_S = K (T) \cap A [[T]] \; .
\]
c) $W_{\rat} (A) = W_{\rat} (K) \cap W (A)$ in $W (K)$. 
\end{theorem}

\begin{proof}
It follows from \cite{CC} Corollaire 4.5 that a) and b) are equivalent. By the definitions, b) implies c). We show that c) implies a). Assume that $x \in K$ is quasi-integral. Then there is some $0 \neq d \in A$ such that $d x^n \in A$ for all $n \ge 0$. Consider
\[
f (T) = \frac{1 - xT + dT^2}{1 - xT} \quad \text{in} \; W_{\rat} (K) \; .
\]
For the coefficients in the power series development
\[
f (T) = \sum^{\infty}_{\nu = 0} b_{\nu} T^{\nu} \quad \text{in} \; W (K)
\]
we find $b_0 = 1 , b_1 = 0 , b_{n+2} = dx^n$ for $n \ge 0$. Hence $f (T) \in W(A)$ and because of c) we have $f (T) \in W_{\rat} (A)$. Now the implication i) $\Rightarrow$ iii) of \cite{CC} Proposition 2.5 implies that $x$ is integral over $A$ since the polynomials $1 - xT + dT^2$ and $1 - x T$ are coprime in $K [T]$. 
\end{proof}

Every Noetherian domain is a Fatou domain by Theorem \ref{t1.3} because if $A [x]$ is contained in a finitely generated $A$-submodule of $K$ it must be finitely generated itself. A different proof based on the definition of Fatou domains can be found in \cite{H} section 3. The following stronger Fatou property will be useful later because it asserts integral coefficients of the canonical irreducible representative of any $f$ in $K (T) \cap A [[T]]$. This allows glueing and shows that the presheaf $U \mapsto W_{\rat} (\Oh_X (U))$ is a sheaf in certain situations. 

\begin{defn}
\label{t1.4}
A domain $A$ is completely integrally closed (or completely normal or cic) if it contains every quasi-integral element of its quotient field $K$.
\end{defn}

By definition, every cic domain is normal.

\begin{theorem}
\label{t1.5}
An integral domain $A$ with quotient field $K$ is completely integrally closed if and only if either of the following three equivalent conditions holds:\\
a) For any element $f \in K (T) \cap A [[T]] \subset K ((T))$ the unique representation $f = P / Q$ with coprime $P , Q \in K [T]$, such that $Q (0) = 1$ satisfies $P , Q \in A [T]$.\\
b) Same condition as in a) but only for $f$ with $\deg P < \deg Q$.\\
c) For any element $f \in W_{\rat} (K) \cap W (A)$ the unique representation $f = P / Q$ with coprime $P , Q \in K [T]$ with $P (0) = 1 , Q (0) = 1$ satisfies $P , Q \in A [T]$. 
\end{theorem}

\begin{proof}
The equivalence of a) with the cic property is \cite{CC} Corollaire 5.5. The equivalence of b) with cic is the main result of \cite{C1}. One direction is easy to see. For $x \in K$ and $0 \neq d \in A$ such that $dx^n\in A$ for all $n \ge 1$, the rational function $f = d / 1 - xT$ is in $A [[T]]$. Condition b) and hence also a) imply that $x \in A$. Hence $A$ is cic. It is clear that a) implies c) As in the proof of Theorem \ref{t1.3}, a consideration of $f (T) = (1 - xT + dT^2) / (1 - xT)$ shows that c) implies that $A$ is cic. 
\end{proof}

Condition b) was introduced by Benzaghou in \cite{B}. Integral domains $A$ which satisfy either a) or b) are also called strong Fatou domains in the literature.

\begin{rem}
What we call Fatou / strong Fatou domains are called fatou / Fatou rings in \cite{CC}.
\end{rem}

\begin{prop} \label{t1.7n}
1) For a ring the properties: strong Fatou domain, cic domain, normal Fatou domain are equivalent.\\
2) Let $A_i \subset K$ be a family of subrings of a field $K$ such that each $A_i$ is a strong Fatou domain. Then $A = \bigcap_i A_i$ is a strong Fatou domain.\\
3) A valuation ring is a strong Fatou domain if and only if it has height one.\\
4) Any intersection of height one valuation rings in a field $K$ is a strong Fatou domain.\\
5) A domain $A$ with quotient field $K$ which is a filtered union of (strong) Fatou domains $A_i$ with quotient fields $K_i$ such that $A_i = A \cap K_i$ for all $i$, is a (strong) Fatou domain. \\
6) Let $A$ be a strong Fatou domain with quotient field $K$. Then the normalization $B$ of $A$ in an algebraic extension field $L$ of $K$ is a strong Fatou domain.
\end{prop}

\begin{proof}
1) By Theorem \ref{t1.5} the conditions cic domain and strong Fatou domain are equivalent. By definition cic domains are normal. By Theorem \ref{t1.3} a normal Fatou domain is a cic domain.\\
2) This follows from the definition of strong Fatou domains or by their characterization as the cic domains.\\
3) This is \cite[Proposition 4]{B}.\\
4) Follows from 3) and 2).\\
5) 
This follows either from the defining property of (strong) Fatou domains or from their characterization in Theorem \ref{t1.3} resp. \ref{t1.5}. \\
6) In terms of the cic property, this is Exercise \S\,1, 14) in \cite{Bou}. Here is the argument: For $y \in L$ and $0 \neq b \in B$ with $by^n \in B$ for all $n \ge 1$ we have to show that $y \in B$. We may assume that $L / K$ is a normal field extension. Since $b$ is integral over $A$ there is a relation $b^d + a_{d-1} b^{d-1} + \ldots + a_0 = 0$ with $d \ge 1 , a_0 , \ldots , a_{d-1} \in A$ and $a = a_0 \neq 0$. It follows that we have $ay^n \in B$ for all $n \ge 1$. Let $P (t) \in K [t]$ be the minimal polynomial of $y$ over $K$. We have
\[
P (t) = t^l + x_{l-1} t^{l-1} + \ldots + x_0 = \prod^l_{i=1} (t - y_i) \quad \text{for} \; x_j \in K \, , \, y_i \in L \; .
\]
For each $y_i$ there is some $\sigma \in G = \Aut_K (L)$ with $y_i = \sigma (y)$. Hence we have $ay^n_i \in B$ for all $n \ge 1$ and $1 \le i \le l$. It follows that the coefficients $x_j$ i.e. $\pm$ the elementary symmetric functions of $y_1 , \ldots , y_l$ satisfy $a^l x^n_j \in B \cap K = A$ for all $n \ge 1$ and $0 \le j \le l-1$. Since $A$ is a cic domain, this implies that $P \in A [t]$ and hence $y \in B$. 
\end{proof}

According to \cite{C1} the class of (strong) Fatou domains is not closed under localization. This leads us to consider the class of Krull domains which have good properties for the discussion of sheaf theoretic properties of $W_{\rat}$. Recall that a Krull domain is an integral domain $A$ for which there exists a set $M$ of discrete valuations $v$ of the quotient field of $A$ such that 1) $A$ is the intersection of the valuation rings of the $v$'s in $M$ and 2) for each $0 \neq a \in A$ there are only finitely many $v \in M$ with $v (a) \neq 0$.

\begin{prop}
\label{t1.6}
1) Any Krull domain is cic or equivalently a strong Fatou domain.\\
2) For a Noetherian domain the properties, Krull, normal, cic $\equiv$ strong Fatou are equivalent.\\
3) The class of Krull domains is stable under localization by multiplicatively closed sets not containing $0$.\\
4) Let $A$ be a Krull domain with quotient field $K$. Let $B$ be the normalization of $A$ in an algebraic extension field $L$ of $K$. Then $B$ is a cic $\equiv$ strong Fatou domain. If $L / K$ is finite, then $B$ is a Krull domain.
\end{prop}

\begin{proof}
1) By definition a Krull domain is in particular the intersection of valuation rings of height one and such intersections are cic domains as mentioned above. \\
2) Krull domains are normal being an intersection of (discrete) valuation rings which are normal. For Noetherian domains the converse is true by \cite{S} Theorem 3.2. We already noted that for Noetherian domains the properties cic and normal are equivalent. \\
3) This is \cite{S} Proposition 4.2.\\
4) For finite extensions $L / K$, the ring $B$ is Krull by \cite{S} Proposition 4.5. In the general case write $L$ as a filtered union of finite extensions $L_i / K$. Then $B$ is the filtered union of the normalizations $B_i$ of $A$ in $L_i$ and we have $B_i = B \cap L_i$. Since the $B_i$ are Krull, they are strong Fatou and by Proposition \ref{t1.7n}, 2) $B$ is strong Fatou as well. Alternatively we can argue with Proposition \ref{t1.7n}, 6).
\end{proof}

Hazewinkel's construction of an $\ind$-scheme representing $W_{\rat}$ on Fatou domains relies on the theory of Hankel determinants and on a theorem of Kronecker. For any ring $A$ and $f = 1 + a_1 T + a_2 T^2 + \ldots$ in $W (A)$ let $H (f) = (a_{i+j})_{i,j \ge 0}$ with $a_0 = 1$ be the corresponding Hankel matrix. The rank of $H (f)$ is defined to be finite if all $(r+1) \times (r+1)$-subminors of $H (f)$ vanish for some $r \ge 1$. In this case all higher minors vanish as well and $\rk H (f)$ is defined to be the smallest such $r$. Otherwiser one sets $\rk H (f) = \infty$. If $A = K$ is a field, it is known that $\rk H (f) = r$ is equivalent to $D_{r-1} \neq 0$ and $D_n = 0$ for $n \ge r$ where $D_n = \det (a_{i+j})_{0 \le i , j \le n}$. Moreover, in this case $\rk H (f)$ is the maximal number of $K$-linearly independent rows (or columns) of $H (f)$. 

\begin{theorem}[Kronecker]
\label{t1.7}
For a field $K$ and $f \in W (K)$ the following assertions are equivalent\\
a) $\rk H (f) = r < \infty$\\
b) $f = P / Q$ with coprime polynomials $P , Q \in K [T]$ such that $r = \max (1 + \deg P , \deg Q)$. 
\end{theorem}

This is \cite{Ko} Theorem 1.2.1 whose proof works for all fields. See also \cite{BB} section 2. The following fact is well known.

\begin{cor}
\label{t1.8}
1) For a field extension $L \supset K$ we have
\[
W_{\rat} (L) \cap W (K) = W_{\rat} (K) \quad \text{in} \; W (L) \; .
\]
2) For an extension of domains $B \supset A$ with $A$ Fatou e.g. Noetherian we have
\[
W_{\rat} (B) \cap W (A) = W_{\rat} (A) \; .
\]
\end{cor}

\begin{proof}
1) For $f \in W_{\rat} (L) \cap W (K)$ the Hankel matrix $H (f)$ has entries in $K$ and by Kronecker's theorem finite rank over $L$, hence over $K$ as well. Applying Theorem \ref{t1.7} again it follows that $f = P / Q$ with $P, Q \in K [T]$ where we may assume that $P (0) = Q (0) = 1$ since $f (0) = 1$.\\
2) Let $L \supset K$ be the quotient fields of $B$ and $A$. Using the Fatou property of $A$ we have
\begin{align*}
W_{\rat} (B) \cap W (A) & \subset W_{\rat} (L) \cap W (A) \\
& = W_{\rat} (L) \cap W (K) \cap W (A) \overset{1)}{=} W_{\rat} (K) \cap W (A) = W_{\rat} (A) \; .
\end{align*}
The other inclusion is clear.
\end{proof}

Here is a result on a Fatou type property where it is not assumed that the rings are domains. It is due to Schoutens \cite{Sch} 4. Corollary. 

\begin{theorem}
\label{t1.9}
For a faithfully flat extension $B \supset A$ of Noetherian rings we have
\[
W_{\rat} (A) = W_{\rat} (B) \cap W (A) \; .
\]
\end{theorem}

The next Corollary will be important later.

\begin{cor}
\label{t1.10}
1) For a field extension $L \supset K$ we have
\begin{equation} \label{eq:5}
W_{\rat} (K) = \Equ (W_{\rat} (L) \rightrightarrows W_{\rat} (L \otimes_K L)) \; .
\end{equation}
2) Let $A \subset B$ be an extension of integral domains with quotient fields $K \subset L$. Assume that $A = B \cap K$ and that $A$ is Fatou or that $B$ is normal or that $B$ is integral over $A$. Then we have
\begin{equation} \label{eq:6}
W_{\rat} (A) = \Equ (W_{\rat} (B) \rightrightarrows W_{\rat} (B \otimes_A B)) \; .
\end{equation}
3) If $A \subset B$ is a faithfully flat extension of Noetherian rings then \eqref{eq:6} holds as well.
\end{cor}

\begin{proof}
1) The morphism $\spec L \to \spec K$ is an fpqc covering and $W$ satisfies the fpqc sheaf property. Hence we have 
\[
W (K) = \Equ (W (L) \rightrightarrows W (L \otimes_K L)) \; .
\]
Hence the right hand side of \eqref{eq:5} is equal to 
\[
W_{\rat} (L) \cap \Equ (W (L) \rightrightarrows W (L \otimes_K L)) = W_{\rat} (L) \cap W (K) \; .
\]
By Corollary \ref{t1.8} this is equal to $W_{\rat} (K)$.\\
2) The inclusion $\subset$ is clear. Using 1), we see that the right hand side of \eqref{eq:6} is contained in
\begin{gather} 
W_{\rat} (B) \cap \Equ (W_{\rat} (L) \rightrightarrows W_{\rat} (L \otimes_K L)) =  W_{\rat} (B) \cap W_{\rat} (K) \notag \\
\subset W (B) \cap W_{\rat} (K) = W (A) \cap W_{\rat} (K) \; . \label{eq:7}
\end{gather}
If $A$ is Fatou, this is contained in $W_{\rat} (K)$ and we are done. For the remaining cases, consider any $f \in W_{\rat} (B) \cap W_{\rat} (K)$ and let $P , Q \in K [T]$ be the unique coprime polynomials with $f = P / Q$ and $P (0) = 1 = Q (0)$. By assumption $f = \tP / \tQ$ with $\tP , \tQ \in B [T] , \tP (0) = 1 = \tQ (0)$. Since $P, Q$ are coprime in $L [T]$ as well, there is a polynomial $H \in L [T]$ with $\tP = HP , \tQ = HQ$ and in particular $H (0) = 1$. For a polynomial $G (T)$ with $G (0) = 1$ consider the normalized polynomial $G^* (T) = T^{\deg G} G (T^{-1})$. The zeroes of $\tQ^*$ are integral over $B$ and since $\tQ^* = H^* Q^*$, the zeroes of $Q^*$ are integral over $B$ as well. If $B$ is normal it follows that $Q^* \in B [T] \cap K [T] = A [T]$ and hence $Q \in A [T]$. The same argument shows that $P \in A [T]$ and therefore $f \in W_{\rat} (A)$. In the case where $B$ is integral over $A$ it follows that the zeroes of $Q^*$ are integral over $A$. By \cite{CC} Lemma 2.6 there is a normalized polynomial $G \in K [T]$ with $GQ^* \in A [T]$ and hence $G^* Q \in A [T]$ where $G^* (0) = 1$. Namely we may write $Q^* = \prod_i (T - \alpha_i)$ where each $\alpha_i$ is a zero of a normalized polynomial $G_i \in A [T]$ and take $G = \prod_i G_i (T) / (T - \alpha_i)$. Hence we may write $f = P_1 / Q_1$ where $P_1 \in K [T] , Q_1 \in A [T]$ with $P_1 (0) = 1 = Q_1 (0)$. Since $f \in W (A)$ by \eqref{eq:7} we find that $P_1 = fQ_1 \in W (A)$ and hence $P_1 \in A [T]$ which implies that $f \in W_{\rat} (A)$. \\
3) The right hand side of \eqref{eq:6} is contained in
\[
W_{\rat} (B) \cap \Equ (W (B) \rightrightarrows W (B \otimes_A B)) = W_{\rat} (B) \cap W (A) \; .
\]
Here we have used that $W$ satisfies the fpqc sheaf condition. Using Theorem \ref{t1.9} we conclude.
\end{proof}
\section{Hazewinkel's $\ind$-scheme $W_J$ is an $\ind$-ring scheme} \label{sec:2}

In \cite{H} Hazewinkel defined a topology on the ring $\Z [X] = \Z [X_1 , X_2 , \ldots ]$ given by a chain of ideals $J_1 \supset J_2 \supset \ldots$ and showed that for any Fatou ring $A$ equipped with the discrete topology we have
\[
\Hom_{\cont} (\Z [X] , A) = W_{\rat} (A) \; .
\]
In this section we recall Hazewinkel's construction and show that the corresponding $\ind$-scheme $W_J = \colim \, \spec (\Z [X] / J_i)$ is an $\ind$ sub-{\it ring} scheme of $W$ invariant under Frobenius and Verschiebung.

We define $\ind$-schemes as in \cite{R}. Let $\Affsch$ be the category of affine schemes. We identify $\Affsch^{\op}$ with Rings, the category of (unital, commutative) rings. Via the Yoneda embedding, the category $\Sch$ of schemes is a full subcategory of the category $\Ph$ of presheaves on $\Affsch$ i.e. of set valued functors on $\Affsch^{\op} =$ Rings. An object $X$ of $\Ph$ is a (strict) $\ind$-scheme if it is isomorphic to a filtered colimit in $\Ph$ of schemes $X_i$ for $i \in I$ where all transition maps $X_i \to X_j$ for $i \le j$ are closed immersions, $X \silo \colim X_i$. Thus for any affine scheme $S$ we have functorial isomorphisms $X (S) \silo \colim_i X_i (S)$. The full subcategory in $\Ph$ of (strict) $\ind$-schemes is called $\Indsch$. Colimits in $\Ph$ exist and are computed pointwise. $\Ind$-schemes satisfy the sheaf condition for the fpqc topology on $\Affsch$, i.e. for all fpqc coverings $\{ S_i \to S \}$ where $S_i$ and $S$ are affine, the following sequence is exact
\begin{equation}
\label{eq:8}
X (S) \longrightarrow \prod_i X (S_i) \rightrightarrows \prod_{i,j} X (S_i \times_S S_j) \; .
\end{equation}
In fact this holds without assuming that the schemes $S$ and $S_i$ are affine if we set $X (Y) = \Mor_{\Indsch} (Y,X)$ for a scheme $Y$, c.f. \cite{R} Lemma 1.4. The category of $\ind$-schemes has fibre products extending those for schemes. If $X = \colim_i X_i$ and $Y = \colim_j Y_j$, the fibre products $X_i \times_S Y_j$ in $\Ph$ are schemes and the transition maps are closed immersions. We have
\begin{equation}
\label{eq:9}
X \times_S Y = \colim_{i,j} X_i \times_S Y_j \quad \text{in} \; \Indsch \; .
\end{equation}
$\Indsch$ has finite products. In particular $X \times Y = X \times_S Y$ for $S = \spec \Z$. 
See \cite{R} Lemma 1.10 for details. Let $(P)$ be a property of morphisms of schemes. A morphism $F \to G$ in $\Ph$ is called representable if for all morphisms $Z \to G$ where $Z$ is an affine scheme, the fibre product $F \times_G Z$ is a scheme. The morphism $F \to G$ has property $(P)$ if it is representable and if all base change morphisms $F \times_G Z \to Z$ have property $(P)$. 

For $A$ in Rings consider the ring $W (A) = 1 + TA [[T]]$ equipped with its Witt addition (= usual multiplication) and its Witt multiplication with unit $1-T$. Moreover $W (A)$ is equipped with the usual Frobenius and Verschiebung endomorphisms $F_N$ and $V_N$ for $N \ge 1$. The ring functor $W$ is represented by $\Z [X] = \Z [X_1 , X_2 , \ldots]$ via the isomorphism $W (A) = \Hom (\Z [X] , A)$ sending $f = 1  + a_1 T + a_2 T^2 + \ldots$ to $\phi_f : \Z [X] \to A$ defined by $\phi_f (X_i) = a_i$. In the other direction $\phi : \Z [X] \to A$ is sent to $f_{\phi} = 1 + \phi (X_1) T + \phi (X_2) T^2 + \ldots$ in $W (A)$. In fact, this is the representation of $W$ viewed as the universal $\lambda$-ring $\Lambda$, but in the notation we do not make that distinction. The element $\xi = 1 + X_1 T + X_2 T^2 + \ldots$ in $W (\Z [X])$ corresponds to $\id_{\Z [X]}$ via $W (\Z [X]) = \Hom (\Z [X] , \Z [X])$. Consider its Hankel matrix $H (\xi) = (X_{i+j})_{i,j \ge 0}$ where $X_0 := 1$. For $n \ge 1$ let $J_n$ be the ideal in $\Z [X]$ generated by all the $(n+1) \times (n+1)$-minors of $H (\xi)$. For $n \ge 0$ set $I_n = (X_{n+1} , X_{n+2} , \ldots)$. Then we have $I_0 \supset I_1 \supset \ldots$ and by the Laplace expansion theorem also $J_1 \supset J_2 \supset \ldots$ and $J_n \subset I_{n-1}$ for all $n \ge 1$. Hazewinkel \cite{H} proves that the ideals $J_n$ are prime and using the existence of integral domains that are not Fatou deduces that the domains $\Z [X] / J_n$ are not Fatou for $n \gg 0$ and in particular not Noetherian. For such $n$ it follows that $\xi_n \in W (\Z [X] / J_n)$ where $\xi_n = 1 + \oX_1 T + \oX_2 T^2 + \ldots$ with $\oX_i = X \mod J_n$ is not in $W_{\rat} (\Z [X] / J_n)$. Note that via $W (\Z [X] / J_n) = \Hom (\Z [X] , \Z [X] / J_n)$ the element $\xi_n$ is the natural projection $\Z [X] \to \Z [X] / J_n$. The ideals $J_n$ define the $J$-topology on $\Z [X]$. A ring homomorphism $\phi : \Z [X] \to A$ to a (discrete) ring $A$ is $J$-continuous if and only if $\phi (J_n) = 0$ for some $n \ge 1$. This means that all $(n+1) \times (n+1)$-minors of $H (\phi (\xi)) = (\phi (X_{i+j}))_{i,j \ge 0}$ vanish.

\textbf{Notation} For a ring $A$ we set $W_J (A) = \{ f \in W (A) \mid \rk H (f) < \infty \}$ and $W^{\le n}_J (A) = \{ f \in W (A) \mid \rk H (f) \le n \}$ for $n \ge 1$ and $W^{< n}_J (A) := W^{\le n-1}_J (A)$ for $n \ge 2$. 

A priori these are subsets of $W (A)$. We will see later that $W_J (A)$ is a \textit{subring} of $W (A)$ for all rings $A$. Clearly, 
\[
W_J (A) = \bigcup_{n \ge 1} W^{\le n}_J (A) \quad \text{and} \quad W^{\le n}_J (A) \subset W^{\le n+1}_J (A) \quad \text{for} \; n \ge 1 \; .
\]
The bijections $W (A) = \Hom (\Z [X] , A)$ restrict to natural bijections
\begin{equation}
\label{eq:10}
W_J (A) = \Hom_{J-\text{cont}} (\Z [X] , A) \quad \text{and} \quad W^{\le n}_J (A) = \Hom (\Z [X] / J_n , A) \; .
\end{equation}
In the language of $\ind$-schemes, this can be rephrased as follows. Consider the integral affine schemes $W^{\le n}_J = \spec \Z [X] / J_n$ and define the $\ind$-scheme $W_J$ by
\[
W_J = \colim_n W^{\le n}_J \quad \text{in} \; \Indsch \; .
\]
Then the $A$-valued points of the schemes $W^{\le n}_J$ resp. the $\ind$-scheme $W_J$ are the sets $W^{\le n}_J (A)$ resp. $W_J (A)$ defined above. We have natural inclusions $W_J (A) \subset W (A)$ and the corresponding morphism $W_J \hookrightarrow W$ in $\Indsch$ is not representable since $W_J$ is not a scheme.

\begin{prop}
\label{t2.1} 
a) For all rings $A$ we have $W_{\rat} (A) \subset W_J (A) \subset W (A)$. \\
b) For Fatou domains $A$ it holds that $W_{\rat} (A) = W_J (A)$.\\
c) Let $A \to A'$ be a map of rings whose kernel $\Nh$ satisfies $\Nh^k = 0$ for some $k \ge 1$. Consider the natural map $W (A) \to W (A') , a \mapsto a'$. If $a' \in W^{< n}_J (A')$ for some $a \in W (A)$, then $a \in W^{< nk}_J (A)$. In particular, the following diagram is cartesian
\[
\xymatrix{
W_J (A) \ar[r]\ar@{^{(}->}[d] & W_J (A') \ar@{^{(}->}[d] \\
W (A) \ar[r] & W (A')
}
\]
d) The $\ind$-scheme $W_J$ is formally smooth over $\spec \Z$. The (non-representable) morphism $W_J \hookrightarrow W$ of presheaves on $\Affsch$ is formally \'etale in the sense that it satisfies the unique lifting property for infinitesimal affine thickenings.\\
e) Let $A$ be a ring whose nilradical $\Nh$ satisfies $\Nh^k = 0$ for some $k \ge 1$, e.g. a Noetherian ring. Then we have a cartesian diagram
\[
\xymatrix{
W_J (A)  \ar[r] \ar@{^{(}->}[d] & W_J (A_{\red}) \ar@{^{(}->}[d]\\
W (A) \ar[r] & W (A_{\red}) \; .
}
\]
\end{prop}

\begin{proof}
a) This is already mentioned in \cite{H} Theorem 4.5. For $f \in W_{\rat} (A) ,f = P / Q$ with $P, Q \in A [T] , P (0) = 1 = Q (0)$, let $A_0$ be the ring generated by the coefficients of $P$ and $Q$. Choose a surjection from a polynomial ring $B = \Z [T_1 , \ldots , T_n] \to A_0$ and lift $f$ to an element $f_B \in W_{\rat} (B)$. Since $B$ is a domain, it follows from Kronecker's theorem \ref{t1.7} that $f_B \in W_J (B)$. The image of $f_B$ under the map $W_J (B) \to W_J (A)$ is $f$ and hence $f \in W_J (A)$.\\
b) This is a result of \cite{H}. For a domain $A$ with quotient field $K$ we have $W_J (A) = W (A) \cap W_J (K)$ because the vanishing of Hankel determinants in $A$ can be checked in $K$. Kroneckers theorem \ref{t1.7} implies that $W_J (K) = W_{\rat} (K)$, and since $A$ is Fatou it follows that $W_J (A) = W (A) \cap W_{\rat} (K) = W_{\rat} (A)$. \\
The proof of c) is an immediate consequence of the following Lemma. Assertions d) and e) follow from c). 
\end{proof}

\begin{lemma}
\label{t2.2}
For the ideals $J_n$ of $\Z  [X]$ defined above, we have the inclusions
\[
J_{nk-1} \subset J^k_{n-1} \quad \text{for} \; n \ge 2 , k \ge 1 \; .
\]
\end{lemma}

\begin{proof}
Consider a matrix $(a_{ij})$ in $M_n (A)$ for some ring $A$ and let $e_1 , \ldots , e_n$ be the canonical basis of $A^n$. Viewing $(a_{ij})$ as an endomorphism $\varphi$ of $A^n$ we have $\varphi (e_j) = \sum_i a_{ij} e_i$. For a subset $J \subset \{ 1 , \ldots , n \}$ of order $r$ set $e_J = e_{j_1} \wedge \ldots \wedge e_{j_r}$ where $j_1 < \ldots < j_r$ are the elements of $J$. Then we have
\[
(\Lambda^r \varphi) (e_J) = \sum^n_{i_1 , \ldots , i_r = 1} a_{i_1 j_1} \cdots a_{i_r j_r} e_{i_1} \wedge \ldots \wedge e_{i_r} = \sum_{|I| = r} a_{IJ} e_I \; .
\]
Here $I$ runs over the subsets of $\{ 1 , \ldots , n \}$ and $a_{IJ}$ is the minor $a_{IJ} = \det (a_{ij})_{i \in I , j \in J}$. Write $n = n_1 + \ldots + n_k$ with $k \ge 1$ and $n_{\nu} \ge 1$. Fix a partition $\{ 1 , \ldots , n \} = J_1 \dcup \ldots \dcup J_k$ with $|J_{\nu}| = n_{\nu}$. Then we have
\begin{align*}
(\Lambda^n \varphi) (e_{\{ 1 , \ldots , n \}} ) & = \pm (\Lambda^{n_1} \varphi) (e_{J_1}) \wedge \ldots \wedge (\Lambda^{n_k} \varphi) (e_{J_k}) \\
& = \pm \sum_{|I_1| = n_1} a_{I_1 J_1} e_{I_1} \wedge \ldots \wedge \sum_{|I_k| = n_k} a_{I_k J_k} e_{I_k} \\
& = \Big( \sum_{\{ 1 , \ldots , n \} = I_1 \dcup \ldots \dcup I_k \atop |I_{\nu}| = n_{\nu}} \pm a_{I_1 J_1} \cdots a_{I_k J_k} \Big) e_{\{ 1 , \ldots , n \} } \; .
\end{align*}
Comparing with the formula $(\Lambda^n \varphi) (e_{ \{ 1 , \ldots , n \} }) = \det (a_{ij}) e_{ \{ 1, \ldots , n \} }$ we find that
\[
\det (a_{ij})  = \sum_{\{ 1 , \ldots , n \} \in I_1 \dcup \ldots \dcup I_k \atop |I_{\nu}| = n_{\nu}} \pm a_{I_1 J_1} \cdots a_{I_k J_k} \; .
\]
It follows that every $nk \times nk$-determinant is a $\Z$-linear combination of $k$-fold products of $n \times n$-subminors. The Lemma follows. 
\end{proof}

For Fatou domains $A$ we have $W_J (A) = W_{\rat} (A)$. Since the domains $\Z [X] / J_n$ are not Fatou for $n \gg 0$ we cannot simply use the universal property to show that $W_J$ has a ring structure. This is one reason for the following considerations.

For any ring $A$ and subset ${\scriptstyle\sum} \subset \spec A$ the canonical maps $A \to A / \ep \to \kappa (\ep) = \Quot (A / \ep)$ induce a ring map
\begin{equation}
\label{eq:11}
W (A) \longrightarrow \prod_{\ep \in \sum} W (\kappa (\ep)) \; , \; f \longmapsto (f_{\ep}) \; .
\end{equation}
Restriction gives maps of sets
\begin{equation}
\label{eq:12}
W^{\le n}_J (A) \longrightarrow \prod_{\ep \in \sum} W^{\le n}_J (\kappa (\ep))
\end{equation}
and
\begin{equation}
\label{eq:13}
W_J (A) \longrightarrow \prod_{\ep \in \sum}\rdp \; W_J (\kappa (\ep)) := \bigcup_{n \ge 1} \prod_{\ep \in \sum} W^{\le n}_J (\kappa (\ep)) \; .
\end{equation}
For a field $K$ and $n \ge 1$ we set
\[
W^{\le n}_{\rat} (K) = \left\{ \begin{array}{rl} f \in W_{\rat} (K)\; \mid & f = P / Q \; \text{for} \; P, Q \in K [T] \; \text{with} \; P (0) = 1 = Q (0) \\
& \text{and} \; \max (1 + \deg P , \deg Q) \le n
\end{array} \right\} \; .
\]
By Kronecker's Theorem \ref{t1.7} we have
\begin{equation}
\label{eq:14}
W^{\le n}_J (K) = W^{\le n}_{\rat} (K) \quad \text{and} \quad W_J (K) = W_{\rat} (K) \; .
\end{equation}
It follows from the definitions, that in $W_{\rat} (K)$ we have
\begin{equation}
\label{eq:15}
W^{\le n}_{\rat} (K) \overset{+}{\cdot} W^{\le m}_{\rat} (K) \subset W^{\le n+m}_{\rat} (K) \quad \text{and} \quad -W^{\le n}_{\rat} (K) \subset W^{\le n+1}_{\rat} (K) \; .
\end{equation}
Moreover
\begin{equation}
\label{eq:16}
F_N (W^{\le n}_{\rat} (K)) \subset W^{\le n}_{\rat} (K) \quad \text{and} \quad V_N (W^{\le n}_{\rat} (K)) \subset W^{\le Nn}_{\rat} (K) \; .
\end{equation}
This follows from the formulas in $W_{\rat} (\oK)$
\begin{equation}
\label{eq:17}
F_N (f) (T^N) = \prod_{\zeta^N = 1} f (\zeta T) \quad \text{and} \quad V_N (f) (T) = f (T^N) \; .
\end{equation}
Setting
\[
\prod_{\ep \in \sum}\rdp W_{\rat} (\kappa (\ep)) := \bigcup_{n \ge 1} \prod_{\ep \in \sum} W^{\le n}_{\rat} (\kappa (\ep))
\]
we therefore see that
\begin{equation}
\label{eq:18}
\prod_{\ep \in \sum}\rdp W_J (\kappa (\ep)) = \prod_{\ep \in \sum}\rdp W_{\rat} (\kappa (\ep)) 
\end{equation}
is a subring of $\prod_{\ep \in \sum} W (\kappa (\ep))$ which is $F_N$- and $V_N$-stable.

If $\sum$ contains all the minimal prime ideals of $A$ then the kernel of the map $A \to \prod_{\ep \in \sum} \kappa (\ep)$ is the nilradical of $A$, hence zero if $A$ is reduced. In this case, the map $W (A) \to \prod_{\ep \in \sum} W (\kappa (\ep))$ is injective and we get the following descriptions of the $A$-valued points of $W^{\le n}_J$ and $W_J$ as subsets of $\prod_{\ep \in \sum} W (\kappa (\ep))$.

\begin{prop}
\label{t2.3}
Let $A$ be a reduced ring. For every subset $\sum$ of $\spec A$ containing the minimal prime ideals of $A$, we have\\
a) $W^{\le n}_J (A) = W (A)\cap \prod_{\ep \in \sum} W^{\le n}_J (\kappa (\ep))$.\\
b) $W_J (A) = W (A) \cap \prod\rdp\!\!\!_{\ep \in \sum} W_J (\kappa (\ep))$.\\
In particular $W_J (A)$ is a subring of $W (A)$.\\
c) $W^{\le n}_J (A) \overset{+}{\cdot} W^{\le m}_J (A) \subset W^{n+m}_J (A)$ and $-W^{\le n}_J (A) \subset W^{\le n+1}_J (A)$. \\
Moreover we have:\\
d) $F_N (W^{\le n}_J (A)) \subset W^{\le n}_J (A)$ and $V_N (W^{\le n}_J (A)) \subset W^{\le Nn}_J (A)$.\\
In particular $W_J (A)$ is $F_N$- and $V_N$-invariant.
\end{prop}

\begin{proof}
For $f \in W (A)$ with $f_{\ep} \in W^{\le n}_J (\kappa (\ep))$ for all $\ep \in \sum$, all $(n+ 1) \times (n+1)$-subminors of $H (f)$ vanish $\mod \ep$ for all $\ep \in \sum$. Hence they vanish in $A$, so that $f \in W^{\le n}_J (A)$. This implies a) from which b) follows by taking unions. Using a), assertion c) follows from \eqref{eq:14} and \eqref{eq:15}. Similarly, assertion d) follows from \eqref{eq:14} and \eqref{eq:16} or in the case of $V_N$ directly from the definition of $W^{\le n}_J (A)$ via Hankel matrices.
\end{proof}

\begin{cor}
\label{t2.4}
Let $A$ be a reduced ring and let $\{ \ep_i \}$ be the set of minimal ideals of $A$ and $\Quot (A)$ its total ring of fractions. Then we have\\
a) $W_J (A) = W (A) \cap \prod_i\rdp W_{\rat} (\kappa (\ep_i))$ in $\prod_i W (\kappa (\ep_i))$.\\
b) $W_J (A) = W (A) \cap W_{\rat} (\Quot (A))$ in $W (\Quot (A))$, if $A$ has only finitely many minimal ideals, e.g. if $A$ is Noetherian.
\end{cor}

\begin{proof}
For a) take $\sum = \{ \ep_i \}$ in Proposition \ref{t2.3}, b) and use \eqref{eq:14}. If $\{ \ep_i \}$ is finite, then $\Quot (A) = \prod_i \kappa (\ep_i)$ and therefore $W_{\rat} (\Quot (A)) = \prod_i W_{\rat} (\kappa (\ep_i)) = \prod_i\rdp W_{\rat} (\kappa (\ep_i))$. Now b) follows from a).
\end{proof}

%
%
%

We now prove that $W_J$ is a ring functor on all commutative rings and not only on the reduced ones as follows from Proposition \ref{t2.3}. Set $R = \Z [X] , R_n = R / J_n$ for $n \ge 1$ and
\[
J_{n,m} = \Ker (R \otimes R \longrightarrow R_n \otimes R_m) \quad \text{where} \; \otimes = \otimes_{\Z} \; .
\]
We do not know if the $J_{n,m}$ are prime ideals. Let $\btu^+ , \btu^{\cdot} : R \to R \otimes R$ be the coaddition and comultiplication on $R$ coming from the functorial ring structure in $A$ on $W (A) = \Hom (R, A) , f \mapsto \phi_f$. Note that the universal polynomials $\btu^+ (X_i)$ and $\btu^{\cdot} (X_i)$ are not the usual ones describing addition and multiplication in the big Witt vector ring $W (A)$, (we have written $W (A)$ for what should more precisely be written $\Lambda (A)$). For example we have the simple formula $\btu^+ (X_i) = \sum_{i+j = n} X_i \otimes X_j$ for $n \ge 1$, where $X_0 := 1$. Let $M : R \to R$ be the antipode for addition and let $\varepsilon_0 , \varepsilon_1 : R \to \Z$ be the cozero and counit. Explicitly $\varepsilon_0 (X_i) = 0$ for $i \ge 1$ and $\varepsilon_1 (X_1) = -1 , \varepsilon_1 (X_i) = 0$ for $i \ge 2$ since we chose $1-T$ as the unit in $W (A)$. Finally let $F_N (\xi) , V_N (\xi) : R \to R$ be the ring maps representing Frobenius and Verschiebung on $W$. Here $\xi = 1 + X_1 t + X_2 t^2 + \ldots$ was the element of $W (R) = \Hom (R,R)$ corresponding to the identity map. With these notations we have the following result.

\begin{lemma}
\label{t2.6}
The ring homomorphisms $\btu^+ , \btu^{\cdot} , M , \varepsilon_0 , \varepsilon_1 , F_N (\xi) , V_N (\xi)$ have the following properties $\btu^+ (J_{n+m}) \subset J_{n,m} , \btu^{\cdot} (J_{n+m}) \subset J_{n,m} , M (J_{n+1}) \subset J_n , \varepsilon_0 (J_n) = 0 , F_n (\xi) (J_n) \subset J_n$ and $V_N (\xi) (J_{nN}) \subset J_n$ for all $n \ge 1$ and $\varepsilon_1 (J_n) = 0$ for $n \ge 2$.
\end{lemma}

\begin{rem}
We suspect that all the $J$-indices in the Lemma are optimal. For $V_N , \varepsilon_0 , \varepsilon_1$ one can check the claims in the Lemma by inspecting the minors of Hankel matrices. For the other ones we will use Proposition \ref{t2.3} and the following fact:
\end{rem}

\begin{lemma}
\label{t2.7}
The rings $R_n \otimes R_m$ are reduced.
\end{lemma}

\begin{proof}
According to \cite{H} Appendix, $J_n$ is a prime ideal. We claim that the integral domain $R_n = \Z [X] / J_n$ has quotient field $K_n$ of characteristic zero. Otherwise we have $p = 0$ in $R_n$ for some prime number $p$, i.e. $p \in J_n$. But $J_n \subset I_{n-1} = (X_n , X_{n+1} , \ldots )$ in $R$ and $p \notin I_{n-1}$ since $R / I_{n-1} \cong \Z [X_1 , \ldots , X_{n-1}]$. As a $\Q$-vector space, $K_m$ is a flat $\Z$-module and hence the injection $R_n \hookrightarrow K_n$ gives an injection $R_n \otimes_{\Z} K_m \hookrightarrow K_n \otimes_{\Z} K_m$. Since $R_n$ is $\Z$-torsionfree hence flat we also get in injection $R_n \otimes_{\Z} R_m \hookrightarrow R_n \otimes_{\Z} K_m$. Thus $R_n \otimes_{\Z} R_m \subset K_n \otimes_{\Q} K_m$. The field $\Q$ being perfect, the algebra $K_n \otimes_{\Q} K_m$ is reduced (Bourbaki, Alg. V, \S\,13 or stacks, Lemma 10.44.3).
\end{proof}

\begin{proofof} \textbf{Lemma \ref{t2.6}}
Let $\xi \in W (R) = \Hom (R,R)$ be the element corresponding to the identity $\id_R$ and let $\alpha , \beta : R \to R \otimes R$ be given by $\alpha (r) = r \otimes 1$ and $\beta (r) = 1 \otimes r$. Viewing $\btu^+ , \btu^{\cdot} : R \to R \otimes R$ as elements of $W (R \otimes R) = \Hom (R , R \otimes R)$ we have the formulas $\btu^+ = W (\alpha) (\xi) + W (\beta) (\xi)$ and $\btu^{\cdot} = W (\alpha) (\xi) \cdot W (\beta) (\xi)$ in $W (R \otimes R)$. This follows immediately from the way addition and multiplication in $W (A)$ are described using $\btu^+$ and $\btu^{\cdot}$. E.g. for $f,g \in W (A)$ the sum $f+g$ corresponds to the composition
\[
f + g : R \xrightarrow{\btu^+} R \otimes R \xrightarrow{f \otimes g} A \otimes A \xrightarrow{\text{multipl.}} A \; .
\]
For $A = R \otimes R$, we take $f = W (\alpha) (\xi) = \alpha \verk \id_R = \alpha$ and $g = W (\beta) (\xi) = \beta$. The formula for $\btu^+$ follows since
\[
(\text{multipl.} \verk (\alpha \otimes \beta)) (r \otimes r') = \alpha (r) \beta (r') = (r \otimes 1) \cdot (1 \otimes r') = r \otimes r' \; .
\]
The antipode $M$ for $\btu^+$ is given by $M = - \xi$ in $W (R)$. Consider the element $\xi_n \in W (R_n)$ corresponding to the natural projection $\pr_n : R \to R_n$. Define $\alpha_n , \beta_m : R_{n+m} \to R_n \otimes R_m$ by $\alpha_n (r) = r \otimes 1$ and $\beta_m (r) = 1 \otimes r$. Set $\btu^+_{n,m} = \pr_{n,m} \verk \btu^+$ and $\btu^{\cdot}_{n,m} = \pr_{n,m} \verk \btu^{\cdot}$ where $\pr_{n,m} = \pr_n \otimes \pr_m : R \otimes R \to R_n \otimes R_m$, and let $M_n = \pr_n \verk M$. By functoriality, we get
\[
\btu^+_{n,m} = W (\alpha_n) (\xi_n) + W (\beta_m) (\xi_m) \quad \text{and} \quad \btu^{\cdot}_{n,m} = W (\alpha_n) (\xi_n) \cdot W (\beta_m) (\xi_m) \quad \text{in} \; W (R \otimes R) \; .
\]
Moreover $M_n = - \xi_n$ in $W (R_n)$ and $\pr_n \verk F_N (\xi) = F_N (\xi_n)$ and $\pr_n \verk V_N (\xi) = V_N (\xi_n)$. Crucially, the element $\xi_n$ is in $W_J (R_n)$, more precisely $\xi_n \in W^{\le n}_J (R_n)$ since $\pr_n$ factors over $R / J_n = R_n$. Since $R_n$ and $R_n \otimes R_m$ are both reduced and since $W^{\le n}_J$ is functorial, assertions c), d) of Proposition \ref{t2.3} imply that
\[
\btu^+_{n,m} , \btu^{\cdot}_{n,m} \in W^{\le n+m}_J (R_n \otimes R_m) \; , \; M_n \in W^{\le n+1}_J (R_n) \; , \; F_n (\xi_n) \in W^{\le n}_J (R_n) 
\]
and $V_N (\xi_n) \in W^{\le Nn}_J (R_n)$. Using the description $W^{\le n}_J (A) = \Hom (R_n , A)$ it follows that
\[
\pr_{n,m} \verk \btu^+ = \btu^+_{n,m} : R \longrightarrow R_n \otimes R_m
\]
factors over $R_{n+m}$. This means that $\btu^+ (J_{n+m}) \subset J_{n,m}$. The other assertions in Lemma \ref{t2.6} follow similarly, the ones for $\varepsilon_0 , \varepsilon_1$ and $V_N$ being also obvious from the definition of $J_n$ via Hankel minors.
\end{proofof}

The following result is a reformulation of Lemma \ref{t2.6}. 

\begin{theorem}
\label{t2.8}
The $\ind$-scheme $W_J$ is a subring in $\Indsch$ of $W$. It is equipped with Frobenius ring endomorphisms $F_N$ and additive Verschiebung endomorphisms $V_N$ which are compatible with those of $W$. More precisely, for the closed affine subschemes $W^{\le n}_J$ of $W_J$ we have factorizations
\[
\xymatrix{
W^{\le n}_J \times W^{\le m}_J \ar[d] \ar[r]^-{+ , \cdot} & W^{\le n+m}_J \ar[d] \\
W_J \times W_J \ar[r]^-{+ , \cdot} & W_J 
}
\qquad
\xymatrix{
W^{\le n}_J \ar[d] \ar[r]^-{-} & W^{\le n+1}_J \ar[d] \\
W_J \ar[r]^{-} & W_J
}
\]
and
\[
\xymatrix{
W^{\le n}_J \ar[d] \ar[r]^{F_N} & W^{\le n}_J \ar[d] \\
W_J \ar[r]^{F_N} & W_J
}
\qquad
\xymatrix{
W^{\le n}_J \ar[d] \ar[r]^{V_N} & W^{\le Nn}_J \ar[d] \\
W_J \ar[r]^{V_N} & W_J \; .
}
\]
\end{theorem}

\begin{cor}
\label{t2.9}
1) The diagram
\[
\xymatrix{
W_J \ar@{^{(}->}[r] \ar[d]_{V_N} & W \ar[d]^{V_N}\\
W_J \ar@{^{(}->}[r] & W
}
\]
in $\Indsch$ is cartesian, i.e. for all rings $A$ we have
\[
V_N (W_J (A)) = W_J (A) \cap V_N (W (A)) \; .
\]
2) For Fatou domains, e.g. for Noetherian domains we have
\[
V_N (W_{\rat} (A)) = W_{\rat} (A) \cap V_N (W (A)) \; .
\]
\end{cor}

\begin{rem}
A different proof of 2) was given in \cite{D} Theorem 1.9 based on the continued fraction algorithm.
\end{rem}

\begin{proof}
2) is an immediate consequence of 1). As for 1) consider $g$ in $W_J (A) \cap V_N (W (A))$. There are $n \ge 1$ and $f = 1 + a_1 T + a_2 T^2 + \ldots$ in $W (A)$ such that $g = f (T^N) \in W^{\le n} (A)$. The Hankel matrix of $g$ is $H (g) = ( a_{\frac{i+j}{N}})_{i,j \ge 0}$ where $a_0 = 1$ and $a_r = 0$ for $r \in \Q \setminus \Z$. Let $I, J \subset \N_0$ be sets with $|I| = n+1 = |J|$. Then the $I \times J$-minor $\det ((a_{i+j})_{i \in I , j \in J})$ of $H (f) = (a_{i+j})$ is equal to the $NI \times NJ$-minor $\det ((a_{\frac{i+j}{N}})_{i \in NI , j \in NJ})$ of $H (g)$ and hence zero since $g \in W^{\le n}_J (A)$. Thus all $(n +1 ) \times (n+1)$-minors of $H (f)$ vanish and therefore $f \in W^{\le n}_J (A)$. The other inclusion is clear.
\end{proof}

The following considerations show that even for very simple non-reduced rings $A$ the ring $W_J (A)$ can be much bigger than $W_{\rat} (A)$. For a ring map $A \to A'$ with kernel $I$ we set $W_J (I) = \ker (W_J (A) \to W_J (A'))$ and similarly for $W^{\le n}_J (I)$ and $W (I)$. Note that $W (I) = 1 + TI [[T]]$. 

\begin{prop}
\label{t2.10}
1) Let $A \twoheadrightarrow A'$ be a surjective ring map with kernel $I$ such that $I^k = 0$ for some $k \ge 1$. Then we have $W (I) = W_J (I) = W^{\le k}_J (I)$ and the sequence
\begin{equation}
\label{eq:19}
0 \longrightarrow W (I) \longrightarrow W_J (A) \longrightarrow W_J (A') \longrightarrow 0 
\end{equation}
is exact.\\
2) Let $A$ be a ring whose nilradical $\Nh$ satisfies $\Nh^k = 0$ for some $k \ge 1$. Then $W_J (\Nh) = W (\Nh)$ and the sequence
\begin{equation}
\label{eq:20}
0 \longrightarrow W (\Nh) \longrightarrow W_J (A) \longrightarrow W_J (A_{\red}) \longrightarrow 0
\end{equation}
is exact.
\end{prop}

\begin{proof}
By the Leibniz formula, the $(k+1) \times (k+1)$-minors of a Hankel matrix $(a_{i+j})_{i,j \ge 0}$ with $a_0 = 1$ and $a_{\nu} \in I$ for $\nu \ge 1$ are sums of products of at least $k$ elements of $I$. Hence they vanish and it follows that $W (I) \subset W^{\le k}_J (I)$. Since the sequence $0 \to W (I) \to W (A) \to W (A') \to 0$ is exact, exactness on the left and in the middle of \eqref{eq:19} follows. Exactness on the right follows from Proposition \ref{t2.1} c).
\end{proof}

\begin{exmp}
Let $B$ be a Fatou domain e.g. a Noetherian domain and $A = B [\varepsilon]$ with $\varepsilon^2 = 0$ the ring of dual numbers over $B$. Then by Propositions \ref{t2.1}, b) and \ref{t2.10} we have exact sequences
\[
0 \longrightarrow 1 + \varepsilon TB [[T]] \longrightarrow W_J (A) \longrightarrow W_{\rat} (A_{\red}) \longrightarrow 0
\]
and
\[
0 \longrightarrow 1 + \varepsilon TB [T] S^{-1} \longrightarrow W_{\rat} (A) \longrightarrow W_{\rat} (A_{\red}) \longrightarrow 0 \; .
\]
Here $S = \{ P \in A [T] \mid P (0) = 1 \}$.

Already for $B = \F_2$ and $A = \F_2 [\varepsilon]$ we see that $W_J (A)$ is uncountable whereas $W_{\rat} (A)$ is countable.
\end{exmp}
\section{Sheaf theoretic properties of $W_{\rat}$} \label{sec:3}

According to Proposition \ref{t1.1n}, the functor $W_{\rat}$ commutes with localisation. However, this does not imply that the presheaf $U \mapsto W_{\rat} (\Oh (U))$ is always a Zariski sheaf - even on an affine scheme. Namely, a relation of the form $f_1 g_1 + \ldots + f_n g_n = 1$ in a ring $A$ usually does not imply a relation of the form $[f_1] h_1 + \ldots + [f_n] h_n$ in $W_{\rat} (A)$ with $h_1 , \ldots , h_n \in W_{\rat} (A)$. E.g. for $f_1 = 2 , f_2 = 3$ in $A = \oZ$, the integral closure of $\Z$ in $\oQ$, we have the relation $2 (-1) + 3 \cdot 1 = 1$ in $A$, but there are no elements $h_1 , h_2 \in W_{\rat} (A) = \uZ A$ with $[2] h_1 + [3] h_2 = 1$ in $W_{\rat} (A)$. It would imply that any $a \in A$ was divisible by either $2$ or $3$. However, there are schemes on which the above presheaf is a Zariski sheaf. We call an integral scheme $X$ strong Fatou if for all open affine subschemes $U = \spec A$ of $X$ the ring $A$ is a strong Fatou domain or equivalently a cic domain, c.f. Theorem \ref{t1.5}, ff.

\begin{prop}
\label{t3.1}
The following conditions on an integral scheme $X$ are equivalent.\\
1) $X$ is strong Fatou.\\
2) For all open subsets $U \subset X$ the ring $\Oh (U)$ is a strong Fatou domain.\\
3) There is a basis $\{ U_i \}_{i \in I}$ of affine open subsets $U_i = \spec A_i$ for which the rings $A_i$ are strong Fatou domains.
\end{prop}

\begin{proof}
The implications 2) $\Rightarrow$ 1) $\Rightarrow$ 3) are clear. To show 3) $\Rightarrow$ 2) let $U = \bigcup_{i \in J} U_i$ for some subset $J \subset I$ be a covering of $U$ by basic opens. Since $\Oh (V) = \bigcap_{x \in V} \Oh_{X,x}$ in the function field $K$ of $X$ for any open $V \subset X$, we have $\Oh (U) = \bigcap_{i \in J} \Oh (U_i)$ in $K$. By Proposition \ref{t1.7n}, 2) it follows that $\Oh (U)$ is a strong Fatou domain.
\end{proof}

\begin{examples} \label{t3.2}
1) A normal locally Noetherian scheme is strong Fatou.\\
2) An integral scheme with an open covering by spectra of Krull domains is strong Fatou.\\
3) The normalization $\tX$ of a strong Fatou scheme $X$ in an algebraic extension $L$ of the function field $K$ of $X$ is strong Fatou.
\end{examples}

\begin{proof}
1) A Noetherian domain is Fatou as mentioned before Definition \ref{t1.4}. Now use Proposition \ref{t1.7n}, 1).\\
2) Recall Proposition \ref{t1.6}. Since the class of Krull domains is stable under localization, the given covering can be refined to a basis by affine opens which are spectra of Krull-, hence strong Fatou domains.\\
3) The morphism $\tX \to X$ is affine. Hence we may assume that $X = \spec A$ where $A$ is a strong Fatou domain. Then $\tX = \spec \tA$ where $\tA$ is the normalization of $A$ in $L$. Now use Proposition \ref{t1.7n}, 6).  
\end{proof}

\begin{theorem}
\label{t3.3}
Let $X$ be a strong Fatou scheme. Then the presheaf of rings $U \mapsto W_{\rat} (\Oh (U))$ is a sheaf for the Zariski topology of $X$. We have $W_{\rat} (\Oh (U)) = W_J (\Oh (U))$ as rings for all open $U \subset X$. 
\end{theorem}

\begin{proof}
Let $j : \eta \to X$ be the inclusion of the generic point $\eta = \spec K$ of $X$. By the Fatou property of $\Oh (U)$, c.f. Theorem \ref{t1.3}, c) the presheaf $U \mapsto W_{\rat} (\Oh (U))$ equals the fibre product of sheaves $W (\Oh_X) \times_{j_* W (\Oh_{\eta})} j_* W_{\rat} (\Oh_{\eta})$. Here $W (\Oh_X)$ is the \textit{sheaf} $U \mapsto W (\Oh_X (U))$. The second assertion holds by Proposition \ref{t2.1}, b) because $\Oh (U)$ is Fatou and the morphism $W_J \to W$ is a morphism of ring schemes, c.f. Theorem \ref{t2.8}.
\end{proof}

\begin{rem}
In \cite[\S\,2]{D} for every scheme $X$, the rational Witt space $W_{\rat} (X)$ was defined using the sheafification of the presheaf $U \mapsto W_{\rat} (\Oh_X (U))$ in the Zariski topology of $X$. All the arithmetic schemes whose rational Witt spaces were studied in \cite{D} are among Examples \ref{t3.2} and hence strong Fatou schemes. Thus for them the sheafification step is not necessary by the theorem. 
\end{rem}

\begin{rem}
The presheaf $X' \mapsto W_{\rat} (\Oh (X'))$ is not a sheaf on the small \'etale site of the spectrum of a field because $W_{\rat}$ does not commute with infinite products. It is a sheaf however on the restricted \'etale site (\textit{finitely presented} \'etale $X' \to X$) if e.g. $X$ is a normal Noetherian scheme. Namely in this case $X'$ is a finite disjoint union of normal Noetherian, hence strong Fatou schemes. Now we can argue similarly as in the proof of Theorem \ref{t3.3}, if we note that $W_{\rat} (\Oh)$ is a sheaf on the restricted \'etale site of a field. Reason: $W_{\rat}$ commutes with finite products and $W_{\rat} (K) = W_{\rat} (K^{\sep})^{G_K}$ c.f. \cite[Proposition 1.3]{D}. 
\end{rem}

Let $\NAS$ be the category of Noetherian affine schemes.

\begin{theorem}
\label{t3.4}
Consider the inclusion of presheaves of rings $W_{\rat} (\Oh (\_ )) \subset W_J (\Oh (\_ ))$ on $\NAS$. Both presheaves satisfy the sheaf condition for $fpqc$-coverings. They agree on integral (Noetherian) schemes but differ already for $X = \spec \F_2 [\varepsilon]$ where $\varepsilon^2 = 0$.
\end{theorem}

\begin{proof}
$W_{\rat}$ commutes with finite direct products of rings. Using Corollary \ref{t1.10}, 3) it follows that $X \mapsto W_{\rat} (\Oh (X))$ satisfies the $fpqc$-sheaf condition on $\NAS$. As mentioned in \S\,\ref{sec:2} the presheaf on $\AS$ defined by the $\ind$-scheme $W_J$ satisfies the $fpqc$-sheaf condition \eqref{eq:8}. Hence the same is true on $\NAS$. Since Noetherian domains $A$ are Fatou, $W_{\rat}$ and $W_J$ agree on the integral schemes in $\NAS$. In the example at the end of section \ref{sec:2}, we have seen that for $A = \F_2 [\varepsilon]$ with $\varepsilon^2 = 0$ the ring $W_{\rat} (A)$ is countable whereas $W_J (A)$ is uncountable.
\end{proof}
\section{Sheafifying the monoid algebra} \label{sec:4}

Consider the map $\omega : \uZ A \to W_{\rat} (A)$ as in \eqref{eq:1}.

\begin{prop}
\label{t4.1}
a) $\omega$ is surjective if and only if every monic polynomial in $A [T]$ decomposes into a product of monic linear factors.\\
b) $\omega$ is injective if and only if factorizations of monic polynomials in $A [T]$ into monic linear factors are unique up to ordering. This is the case if and only if $A$ is a domain.
\end{prop}

\begin{proof}
Replacing $P (T)$ by $P^* (T) = T^{\deg P} P (T^{-1})$ interchanges monic polynomials satisfying $P (0) \neq 0$ with polynomials whose constant term is $1$. Now a) and the first part of b) follow by elementary arguments from the definition of $\omega$. For an integral domain $A$ with quotient field $K$ we have $\uZ A \subset \uZ K \subset W_{\rat} (K)$. The latter inclusion holds because factorization in $K [T]$ is unique up to ordering. Hence $\uZ A \subset W_{\rat} (A)$ follows. If $0 \neq a \in A$ is a zero-divisor there is some $0 \neq b \in A$ with $a \cdot b = 0$. Hence we have the two distinct factorizations $T (T - (a+b)) = (T-a) (T-b)$. In this case $1 - (a+b) T = (1-aT) (1 - bT)$ i.e. the non-zero element $(a+b) - (a) - (b) \in \uZ A$ is in the kernel of $\omega$. 
\end{proof}

In the following we discuss the sheafification of the map $\omega$ in various Grothendieck topologies. We follow the conventions of \cite[7.6]{stacks} regarding sites and refer to \cite[7.10]{stacks} for the $2$-step process of sheafification $F \mapsto F^+ \mapsto F^{\sharp} = F^{++}$. For the underlying category of the site we will take a small version as in \cite{stacks} of either $\AS$ or $\NAS$. There are various ways to sheafify the reduced monoid-algebra construction and the results are the same. For a commutative ring $R$ we write $\uR A$ or $\uR (A)$ for the quotient of the monoid algebra of $(A , \cdot)$ with coefficients in $R$ by the ideal $R (0)$. 

\begin{prop}
\label{t4.2}
Consider the following presheaves on $\Ch = \NAS$ or $\AS$ $F_1 (Y) = \uZ (\Oh (Y)) , F_2 (Y) = \underline{\Z^{\sharp} (Y)} (\Oh (Y)) , F_3  (Y) = \uZ (\Oh^{\sharp} (Y)) , F_4 (Y) = \underline{\Z^{\sharp} (Y)} (\Oh^{\sharp} (Y))$. Here $\sharp$ denotes sheafification in any not necessarily subcanonical topology on $\Ch$. The restriction maps are the obvious ones. Then the natural maps of presheaves of rings
\[
\xymatrix{
 & F_2 \ar[dr] & \\
F_1 \ar[ur] \ar[dr] & & F_4\\
 & F_3 \ar[ur] & 
}
\]
induce isomorphisms of sheaves
\[
\xymatrix{
 & F^{\sharp}_2 \ar[dr]^{\sim} & \\
F^{\sharp}_1 \ar[ur]^{\sim} \ar[dr]^{\sim} & & F^{\sharp}_4\\
 & F^{\sharp}_3 \ar[ur]^{\sim} & 
}
\]
Moreover we have an isomorphism $F^+_1 \silo F^+_2$.
\end{prop}

The proof is straightforward. Sheafification can be done either in the category of set-, group- or ring valued presheaves and the results are the same. In \cite{V} Proposition 2.1.1. a version of $F_2$ is sheafified in a related context. We will usually sheafify $F_1$, i.e. the presheaf $\uZ (\Oh) : Y \mapsto \uZ \Oh (Y)$. Let $W_{\rat} (\Oh)$ be the presheaf $Y \mapsto W_{\rat} (\Oh (Y))$. 

\textbf{Convention} We write
\[
\uZ (A)^{\sharp} := \uZ (\Oh)^{\sharp} (\spec A) \quad \text{and} \quad W_{\rat} (A)^{\sharp} = W_{\rat} (\Oh)^{\sharp} (\spec A) \; .
\]

\begin{prop}
\label{t4.3}
Consider a pretopology on $\Ch = \AS$ or $\Ch = \NAS$ whose coverings include at least the syntomic finite locally free faithfully flat morphisms $Y \to X$. Then the map of presheaves $\omega : \uZ (\Oh) \to W_{\rat} (\Oh)$ induces a surjection of sheaves $\uZ (\Oh)^{\sharp} \twoheadrightarrow W_{\rat} (\Oh)^{\sharp}$.
\end{prop}

\begin{proof}
According to \cite[Lemma 10.136.14]{stacks}, given a monic polynomial $P (T)$ in $A [T]$ there exists a syntomic finite locally free faithfully flat morphism $\spec B \to \spec A$ such that the image of $P (T)$ in $B [T]$ decomposes into a product of monic linear factors. This implies the corresponding assertion for polynomials with constant term $1$ and the claim follows. 
\end{proof}

We now discuss injectivity of the map $\uZ (\Oh) \to W_{\rat} (\Oh)$ after sheafification. By Proposition \ref{t4.1} b), non-injectivity is caused by zero divisors. For some of them, their influence disappears under sheafification. For example an idempotent $e$ is a zero-divisor and correspondingly the map
\[
(\uZ A_1 \otimes \uZ A_2) / I = \uZ (A_1 \times A_2) \longrightarrow W_{\rat} (A_1 \times A_2) = W_{\rat} (A_1) \times W_{\rat} (A_2)
\]
is not injective. Here $I = \Z (0_{A_1}) \otimes \uZ A_2 + \uZ A_1 \otimes \uZ (0_{A_2})$ in $\Z A_1 \otimes \Z A_2$. However, after sheafification in any topology finer than the Zariski topology, since $\spec (A_1 \times A_2) = \spec A_1 \amalg \spec A_2$ we have
\[
\uZ (A_1 \times A_2)^{\sharp} = \uZ (A_1)^{\sharp} \times \uZ (A_2)^{\sharp} \; .
\]
Thus the map $\uZ (A_1 \times A_2)^{\sharp} \to W_{\rat} (A_1 \times A_2)^{\sharp}$ is injective if the maps $\uZ (A_i)^{\sharp} \to W_{\rat} (A_i)^{\sharp}$ are injective for $i = 1,2$. For other types of zero-divisors however, the non-injectivitiy can not be sheafified away in subcanonical topologies, as the following example shows.

\begin{example} 
\label{t4.4}
For $A = \F_2 [\varepsilon] , \varepsilon^2 = 0$ and any subcanonical topology on $\NAS$ or $\AS$ the map $\uZ (A)^{\sharp} \to W_{\rat} (A)^{\sharp}$ is not injective.
\end{example}

\begin{proof}
Since $Y \mapsto W (\Oh (Y))$ is a sheaf for any subcanonical topology we have a commutative diagram, where $A = \F_2 [\varepsilon]$ and where on the right we have written the images of the element $a = (\varepsilon) \in \uZ (A)$
\[
\xymatrix{
\uZ (A) \ar[r] \ar[d] & W_{\rat} (A) \ar@{^{(}->}[r] \ar[d] & W (A) \ar@{=}[d] \\
\uZ (A)^{\sharp} \ar[r] & W_{\rat} (A)^{\sharp} \ar@{^{(}->}[r] & W (A)
} \qquad
\xymatrix{
a \ar@{|->}[r] \ar@{|->}[d] & a_{\rat} \ar@{|->}[r] \ar@{|->}[d] & a_W \ar@{=}[d] \\
a^{\sharp} \ar@{|->}[r] & a^{\sharp}_{\rat} \ar@{|->}[r] & a_W\; .
}
\]
The canonical multiplicative map $[\;] : A \to W (A)$ is injective and hence $a_W = [\varepsilon]$ is non-zero. It follows that $a^{\sharp} \in \uZ (A)^{\sharp}$ is non-zero as well. The $2$-multiplication map on the presheaf $\uZ (\Oh)$ is injective and hence the $2$-multiplication on $\uZ (\Oh)^{\sharp}$ is injective as well. Thus we have $0 \neq 2 a^{\sharp}$ in $\uZ (A)^{\sharp}$. On the other hand, $2a_{\rat} = (1 - \varepsilon t)^2 = 1 \ent 0$ in $W_{\rat} (A)$. Hence the image $2a^{\sharp}_{\rat}$ of $2a_{\rat}$ in $W_{\rat} (A)^{\sharp}$ is zero as well. Since $0 \neq 2a^{\sharp} \mapsto 2a^{\sharp}_{\rat} = 0$, it follows that the map $\uZ (A)^{\sharp} \to W_{\rat} (A)^{\sharp}$ is not injective.
\end{proof}

\begin{prop}
\label{t4.5}
Consider either of the sites $\Ch = \NAS$ or $\Ch = \AS$. For $\spec A$ in $\Ch$ assume that every covering $\{ \spec A_i \to \spec A \}_{i \in I}$ has a refinement to a covering $\{ \spec B_j \to \spec A \}_{j \in J}$ where the $B_j$ are integral domains. Then the map $\uZ (A)^{\sharp} \to W_{\rat} (A)^{\sharp}$ is injective.
\end{prop}

\begin{proof}
For a map $F \to G$ of abelian presheaves on a site we have $\Ker (F^{\sharp} \to G^{\sharp}) = \Ker (F \to G)^{\sharp}$. Hence, for every element $a^{\sharp} \in \Ker (\uZ (A)^{\sharp} \to W_{\rat} (A)^{\sharp})$ there exist a covering $\{ \spec A_i \to \spec A \}$ and elements $a_i \in \Ker (\uZ (A_i) \to W_{\rat} (A_i))$ which map to $a^{\sharp} \, |_{\spec A_i}$. For a refinement as in the assumption we have a map of index sets $J \to I , j \mapsto i (j)$ and ring homomorphisms $\varphi_j : A_{i (j)} \to B_j$ for $j \in J$ such that the diagrams
\[
\xymatrix{
A \ar[rr] \ar[dr] & & A_{i (j)} \ar[dl]^{\varphi_j} \\
 & B_j &
}
\]
commute. The element $a^{\sharp}$ restricts to elements $a^{\sharp} \, |_{\spec B_j}$ in $\Ker (\uZ (B_j)^{\sharp} \to W_{\rat} (B_j)^{\sharp})$ which are represented by $\varphi_j (a_{i (j)})$ in $\Ker (\uZ (B_j) \to W_{\rat} (B_j))$. By Proposition \ref{t4.1}, b) the latter groups are zero since the $B_j$ are integral domains. Hence $a^{\sharp} \, |_{\spec B_j} = 0$ for all $j$ and since $\{ \spec B_j \to \spec A \}$ is a covering, it follows that $a^{\sharp} = 0$.
\end{proof}

Using the preceding facts it follows that in some situations, the rational Witt vectors are obtained by sheafifying the monoid algebra. The finite flat (or $fp$-) site on $\NAS$ has coverings $\{ X_i \to X \}$ by finite flat morphisms which are jointly surjective. Every finite flat morphism between Noetherian schemes is an $fppf$-morphism and in particular (universally) open. Since a finite morphism is (universally) closed it follows that for connected $X$ any $fp$-covering $\{ X_i \to X \}$ has a refinement $\{ Y \to X \}$ where $Y \to X$ is one of the morphisms $X_i \to X$. Note that by faithful flatness, if $\{ \spec B \to \spec A \}$ is an $fp$-covering, the map $A \to B$ is injective. 

\begin{cor}
\label{t4.6}
Consider the site $\NAS$ with the $fp$-pretopology. For any Dedekind ring $A$ we have
\[
W_{\rat} (A) = \uZ (A)^{\sharp} \; .
\]
\end{cor}

\begin{rem}
In particular, we have $W_{\rat} (K) = \uZ (K)^{\sharp}$ for all fields $K$. Recall Theorem \ref{t1.1}, (2) for $A_0 = K$ for an explicit description of $W_{\rat} (K)$ for an arbitrary field, a result from \cite{BV} \S\,2.
\end{rem}

\begin{proof}
The canonical map $\uZ (\Oh)^{\sharp} \to W_{\rat} (\Oh)^{\sharp}$ is surjective by Proposition \ref{t4.3}. To show that $\uZ (A)^{\sharp} \to W_{\rat} (A)^{\sharp}$ is an isomorphism for Dedekind rings $A$ we apply Proposition \ref{t4.5}. Let $\{ \spec A_i \xrightarrow{\pi_i} \spec A \}$ be an $fp$-covering of $\spec A$. We have to find a refinement which is an $fp$-covering of $\spec A$ by integral affine schemes. Since $\spec A$ is connected, $\{ \spec B \xrightarrow{\pi} \spec A \}$ will be an $fp$-covering for $B = A_i , \pi = \pi_i$ for some $i$. The map $A \to B$ is injective. We claim that the refinement $\{ \spec B^{\red} \to \spec A \}$ is an $fp$-covering as well. Since $\spec B^{\red} \to \spec B$ is finite surjective, we only need to check flatness. Over Dedekind rings, flatness is equivalent to being torsion-free. Let $a \ob = 0$ in $B^{\red}$ for elements $0 \neq a \in A$ and $\ob = b + \Nh \in B^{\red} = B / \Nh$ where $b \in B$ and $\Nh$ is the nilradical of $B$. Then $ab \in \Nh$ and hence $a^N b^N = 0$ for some $N \ge 1$. Since $A$ is reduced, $a^N \neq 0$ and hence $b^N = 0$ since $B$ is $A$-torsion free as a flat $A$-module. Thus $b \in \Nh$ and therefore $\ob = 0$. Hence $B^{\red}$ is $A$-torsion free, hence a flat $A$-module as remarked above. Thus we may assume that the ring $B$ in our $fp$-covering $\{ \spec B \to \spec A \}$ is reduced. Let $\eg$ be a prime ideal of $B$ that is mapped to the generic point $(0) \in \spec A$ and let $\spec C$ be an irreducible component with the reduced structure of $\spec B$ which contains $\eg$. Thus $C = B / \ep$ for some minimal prime ideal $\ep \subset \eg$. We have $\ep \cap A \subset \eg \cap A = (0)$ and hence $\ep \cap A = 0$. Therefore the generic point of the integral Noetherian scheme $\spec C$ is mapped to the generic point of $\spec A$. Since the composition $\spec C \to \spec B \to \spec A$ is finite, the image of $\spec C \to \spec A$ is closed and therefore all of $\spec A$. The map $A \to C$ is injective since $\ep \cap A = 0$, hence $C$ is $A$-torsion free and hence a flat $A$-module. Hence we have refined our original $fp$-covering by the $fp$-covering $\{ \spec C \to \spec A \}$ in $\NAS$ where $C$ is an integral domain. Now the proposition follows from Proposition \ref{t4.5}.
\end{proof}

For suitable non-subcanonical topologies, the map $\uZ (\Oh)^{\sharp} \to W_{\rat} (\Oh)^{\sharp}$ is injective and even an isomorphism as we will now see. However $W_{\rat} (A)^{\sharp}$ will differ from $W_{\rat} (A)$ in general.

The coverings of the finite (or $f$-)pretopology on a category of schemes are the jointly surjective families of finite morphisms.

\begin{cor}
\label{t4.7}
For any pretopology on $\NAS$ or $\AS$ which is finer than the $f$-pretopology, the map $\uZ (\Oh)^{\sharp} \silo W_{\rat} (\Oh)^{\sharp}$ is an isomorphism of ring sheaves.
\end{cor}

For example, the corollary applies to the proper-, the finite-, the $qfh$- and the $h$-site, c.f. \cite[Definition 2.5]{GK}.

\begin{proof} Surjectivity follows from Proposition \ref{t4.3}. Injectivity is a consequence of Proposition \ref{t4.5}, as follows. Let $\{ X_i \to X \}$ be a covering. Consider the irreducible components $X_{ij}$ of $X_i$. If we equip $X_{ij}$ with the reduced scheme structure then $(X_{ij} \to X_i)_j$ is a jointly surjective family of closed immersions. By assumption it is therefore a covering of $X_i$ and hence $\{ X_{ij} \to X_i \to X \}_{i,j}$ is a covering in the given pretopology. Since the $X_{ij}$ are integral schemes, injectivity follows from Proposition \ref{t4.5}. 
\end{proof}

\begin{rem} For the categories of (affine) schemes considered in \cite[Lemma 3.2]{GK} the points for the $f$-topology are the schemes of the form $P = \spec B$ for an $aic$ domain $B$. Since the functors $A \mapsto \uZ (A)$ and $A \mapsto W_{\rat} (A)$ commute with filtered colimits, the stalks of $\uZ (\Oh)^{\sharp}$ and $W_{\rat} (\Oh)^{\sharp}$ in the point $P$ are given by
\[
\uZ (\Oh)^{\sharp}_P = \uZ B = W_{\rat} (B) = W_{\rat} (\Oh)^{\sharp}_P \; .
\]
Hence the natural map $\uZ (\Oh)^{\sharp} \to W_{\rat} (\Oh)^{\sharp}$ is an isomorphism on the stalks for the $f$-topology, hence an isomorphism of $f$-sheaves and therefore an isomorphism for any finer (pre-)topology.
\end{rem}

Let $W_J (\Oh)$ and $W (\Oh)$ be the pre-sheaves of rings $Y \mapsto W_J (\Oh (Y))$ c.f. Theorem \ref{t2.8} and $Y \mapsto W (\Oh (Y))$. 

On the category $\NAS$ the coverings for the $f$-pretopology can always be refined by finite coverings since a Noetherian scheme has only finitely many irreducible components. 

\begin{prop}
\label{t4.8}
1) For any pretopology on $\NAS$ for which all finite families of jointly surjective closed immersions are coverings, the map $W_{\rat} (\Oh)^{\sharp} \silo W_J (\Oh)^{\sharp}$ is an isomorphism of ring-sheaves.\\
2) For any pretopology on $\NAS$ which is finer than the $f$-pretopology we have canonical isomorphisms of subring-sheaves of $W (\Oh)^{\sharp}$
\[
\uZ (\Oh)^{\sharp} \silo W_{\rat} (\Oh)^{\sharp} \silo W_J (\Oh)^{\sharp} \; .
\]
\end{prop}

\begin{proof}
1) By Proposition \ref{t2.1} a), the canonical map of presheaves $W_{\rat} (\Oh) \to W_J (\Oh)$ is injective. Hence $W_{\rat} (\Oh)^{\sharp} \to W_J (\Oh)^{\sharp}$ is injective as well. For $X = \spec A$ and $s \in W_J (\Oh)^{\sharp} (X)$ choose a covering $\{ X_i \to X \}$ such that each restriction $s \, |_{X_i} \in W_J (\Oh)^{\sharp} (X_i)$ is the image of an element $s_i \in W_J (A_i)$ where $X_i = \spec A_i$. For each $i$, the family $\{ X_{ij} \to X_i \}_j$ as in the proof of Corollary \ref{t4.7} is finite and by assumption therefore a covering. Hence we obtain a covering $\{ X_{ij} \to X_i \to X \}_{i,j}$ such that each $X_{ij} = \spec A_{ij}$ is an integral scheme. The restrictions $s \, |_{X_{ij}} \in W_J (\Oh)^{\sharp} (X_{ij})$ are represented by the images $s_{ij}$ of $s_i$ under the map $W_J (A_i) \to W_J (A_{ij})$ induced by $A_i \to A_{ij}$. As a Noetherian integral domain, $A_{ij}$ is Fatou and hence we have $W_{\rat} (A_{ij}) = W_J (A_{ij})$ by Proposition \ref{t2.1}, b). Hence every $s \, |_{X_{ij}}$ is the image of a section of $W_{\rat} (\Oh) (X_{ij})$ and hence of $W_{\rat} (\Oh)^{\sharp} (X_{ij})$. Thus the map $W_{\rat} (\Oh)^{\sharp} \to W_J (\Oh)^{\sharp}$ is surjective.\\
2) follows from 1) and Corollary \ref{t4.7}. 
\end{proof}

We now discuss Frobenius and Verschiebung maps on the sheafified reduced monoid algebra $\uZ (\Oh)^{\sharp}$ and compare with those on the rational Witt vector sheaves.

Frobenius is easy, for a scheme $X$, there is a map of multiplicative monoids
\[
F_N : \Oh (X) \longrightarrow \uZ \Oh (X) \; , \; F_N (f) = (f^N) \; .
\]
It induces a homomorphism of presheaves of rings $F_N : \uZ (\Oh) \to \uZ (\Oh)$ on any category of schemes and hence a ring map of the associated sheaves $F_N : \uZ (\Oh)^{\sharp} \to \uZ (\Oh)^{\sharp}$ for any (pre-)topology on the given category of schemes. Under the canonical maps $\uZ \Oh (X) \to W_{\rat} (\Oh (X))$ and $\uZ (\Oh)^{\sharp} \to W_{\rat} (\Oh)^{\sharp}$ the endomorphisms $F_N$ are compatible with the corresponding Frobenius endomorphisms on rational Witt vector (pre-)sheaves. 

Verschiebung is less obvious. For example, for all the sites in Proposition \ref{t4.3}, the map of sheaves $\uZ (\Oh)^{\sharp} \twoheadrightarrow W_{\rat} (\Oh)^{\sharp}$ is surjective. It is not clear to me if e.g. for the $fp$- or $fppf$-topology, the (sheafified) Verschiebung $V_N$ on $W_{\rat} (\Oh)^{\sharp}$ can be lifted to an endomorphism of $\uZ (\Oh)^{\sharp}$. If the (pre-)topology on $\NAS$ or $\AS$ is finer than the $f$-pretopology, then by Corollary \ref{t4.7} we have an isomorphism $\omega : \uZ (\Oh)^{\sharp} \silo W_{\rat} (\Oh)^{\sharp}$ and hence there are additive Verschiebung endomorphisms $V_N$ on $\uZ (\Oh)^{\sharp}$. They have the following explicit description: It suffices to give the map of presheaves of sets $V_N : \Oh \to \uZ (\Oh)^{\sharp}$ because this induces a map of presheaves of abelian sheaves $\uZ (\Oh) \to \uZ (\Oh)^{\sharp}$ and hence a homomorphism of abelian sheaves $V_N : \uZ (\Oh)^{\sharp} \to \uZ (\Oh)^{\sharp}$. Thus let $X = \spec A$ be an affine scheme and let $a \in \Oh (X) = A$ be an element. For the polynomial $P (T) = T^n-a$ choose a finite surjective morphism $\spec B \to \spec A$ such that we have
\[
T^n - a = \prod^n_{i=1} (T - \alpha_i) \quad \text{in} \; B [T] \; .
\]
For example, we can take $B = A [X,Y] / (X^n - 1 , Y^n- a)$ or use \cite[Lemma 10.136.14]{stacks}. Then $V_N (a) \in \uZ (A)^{\sharp}$ is the image of $\sum_i (\alpha_i) \in \uZ B$ under the canonical map $\uZ B \to \uZ (B)^{\sharp}$. A priori, $V_N (a) \in \uZ (B)^{\sharp}$ but comparison with $V_N$ on $W_{\rat}$ shows that $V_N (a) \in \uZ (A)^{\sharp}$. Namely, we have
\[
V_N (1 - aT) = 1 - aT^n = \prod^n_{i=1} (1 - \alpha_i T) \quad \text{in} \; W_{\rat} (B) \; .
\]
Under the isomorphism $\omega : \uZ (B)^{\sharp} \silo W_{\rat} (B)^{\sharp}$ we therefore have $\omega (V_N (a)) = V_N (\omega (a))$. Since $\omega (a) \in W_{\rat} (A)^{\sharp}$ and since $V_N$ maps $W_{\rat} (A)^{\sharp}$ into itself, it follows that $V_N (a) \in \uZ (A)^{\sharp} \subset \uZ (B)^{\sharp}$. It is also possible to show directly i.e. without using $V_N$ on $W_{\rat}$ that
\[
V_N (a) \in \mathrm{Equ} (\uZ (B)^{\sharp} \rightrightarrows \uZ (B \otimes_A B)^{\sharp}) = \uZ (A)^{\sharp}
\]
and that $V_N (a)$ is independent of the choice of the finite surjective morphisms $\spec B \to \spec A$. 

We end this section with some remarks about connections to Suslin's and Voevodsky's work. For a scheme $X$, Voevodsky has studied the $h$- and $qfh$-sheaves of abelian groups associated to the presheaf of abelian groups $Y \mapsto \Z X (Y)$, c.f. \cite[\S\,2]{V}, in this regard note Proposition \ref{t4.2}. For the scheme $X = \A^1$ we have $\Z \A^1 (Y) = \Z \Oh (Y)$ and Voevodsky's sheaves $\Z_h (\A^1)$ resp. $\Z_{qfh} (\A^1)$ become $\Z (\Oh)^{\sharp}$ for $\sharp = h$ resp. $\sharp = qfh$ in our notation, considered on the category of separated schemes of finite type over a Noetherian excellent scheme $S$. Since $\A^1$ is affine, \cite[Theorem 3.3.5]{V} implies that $\Z_h (\A^1) = \Z_{qfh} (\A^1)$ i.e. that $\Z (\Oh)^{\sharp}$ is the same sheaf for $\sharp = h$ and $\sharp = qfh$. It follows that on $\NAS$ the associated sheaf $W_{\rat} (\Oh)^{\sharp}$ is the same for $\sharp = h$ and $\sharp = qfh$ since we have $\uZ (\Oh)^{\sharp} = W_{\rat} (\Oh)^{\sharp}$ for those topologies by Corollary \ref{t4.7}. 

Let $Sch / k$ resp. $Nor / k$ be the category of separated (normal) schemes of finite type over a field $k$ of characteristic exponent $p \ge 1$ with the $qfh$-topology. We claim that the presheaf $Y \mapsto W_{\rat} (\Oh (Y)) \otimes \Z [1/p]$ is a $qfh$-sheaf on the category $Nor / k$ in the sense of \cite[Definition 6.1]{SV1}. Namely, by Theorem \ref{t3.3} it is a Zariski sheaf on each $S$ in $Nor / k$ which is condition (6.1.1) in \cite{SV1}. Moreover it satisfies Galois descent in the sense of \cite[(6.1.2)]{SV1} because of Theorem \ref{t1.1}. It follows from \cite[\S\,6]{SV1}, that for any $S = \spec A$ in $Nor / k$ the sections over $S$ of the $qfh$-sheaf attached to $\uZ (\Oh) \otimes \Z [1 / p]$ are given by the formula
\begin{equation}
\label{eq:21}
(\uZ (\Oh) \otimes \Z [ 1/p ])^{\sharp} (S) \overset{\text{Cor. \ref{t4.7}}}{=} (W_{\rat} (\Oh) \otimes \Z [ 1/p ] )^{\sharp} (S) = W_{\rat} (A) \otimes \Z [ 1/p ] \; .
\end{equation}
On the other hand, by \cite[Theorem 6.7]{SV1}, we have
\[
(\uZ (\Oh) \otimes \Z [1/p])^{\sharp} (S) = z^c_0 (\A^1) (S) \; .
\]
Here $z^c_0 (X) (S)$ for $X$ in $Sch / k$ is the free $\Z [1/p]$-module generated by closed integral subschemes $Z \subset X \times S$ for which the projection $Z \to S$ is a finite surjective morphism.

It follows that $W_{\rat} (A) \otimes \Z [1/p] = z^c_0 (\A^1) (\spec A)$ for normal Noetherian domains over a field. This formula motivated the next section and in particular Theorem \ref{t5.1} where a direct proof of a more general result is given. 

We also point out that by \cite[Lemma 7.2]{SV1} we have as $qfh$-sheaves on $Sch/ k$ 
\[
\uExt^i (\uZ (\Oh)^{\sharp} , \Z / n) = \uExt^i (W_{\rat} (\Oh)^{\sharp} , \Z / n) = 0 \quad \text{for} \; p \nmid n \; \text{and} \; i \ge 0 \; .
\]
\section{Rational Witt vectors and finite correspondences} \label{sec:5}

For a morphism $X \to S$ of separated Noetherian schemes, Suslin and Voevodsky define the subgroup $c_{\equi} (X / S , r)$ of universally integral cycles on $X$ relative to $S$ which are proper and equidimensional of dimension $r$ over $S$ and satisfy the equivalent conditions of \cite[Lemma 3.3.9]{SV2}. The general definition is involved but if $S$ is a normal Noetherian scheme and $X / S$ is smooth of finite type of equidimension $d$ over $S$, then by \cite[Proposition 3.4.8]{SV2}, we have
\begin{equation}
\label{eq:22}
c_{\equi} (X / S , d-1) = \PropCycl_{\equi} (X / S , d-1) = \PropCaDir (X / S) \; .
\end{equation}
The general definition of the group in the middle is given in \cite[Definition 3.1.3]{SV2}. For $X / S$ of finite type and $S$ geometrically unibranch e.g. $S$ normal its description simplifies: The abelian group $\PropCycl_{\equi} (X / S , r)$ is freely generated by cyles of integral closed subschemes $Z$ in $X$ which are proper and equidimensional of dimension $r$ over $S$, c.f. \cite[Corollary 3.4.3]{SV2}. Finally, $\PropCaDir (X/S)$ is the group of relative Cartier divisors on $X$ which are proper over $S$.

For general separated Noetherian $S$-schemes $X , Y$ one sets
\[
\Corr_S (X,Y) = c_{\equi} (X \times_S Y / X ,0) \; , \; \text{c.f. \cite[Definition 1A.9]{MVW}} \; .
\]
Using the good funtoriality properties of universally integral cycles, one may compose these correspondences, \cite[Defintion 1A.11]{MVW}. The graph of any morphism $f : X \to Y$ over $S$ gives an element $\Gamma_f$ of $\Corr_S (X,Y)$ and one has $\Gamma_{g \verk f} = \Gamma_f \verk \Gamma_g$ if $g : Y \to Z$ is another morphism. For $S = \spec \Z$ we write $\Corr (X,Y) = \Corr_S (X,Y)$. For a Noetherian separated scheme $X$ and $\A = \A^1_{\spec \Z}$ consider the map of sets
\begin{equation}
\label{eq:23}
\Oh (X) = \Mor (X , \A) \longrightarrow \Corr (X , \A) \quad \text{where} \; f \mapsto \Gamma_f \; .
\end{equation}
The map \eqref{eq:23} induces an additive homomorphism 
\begin{equation}
\label{eq:24}
\Z \Oh (X) \longrightarrow \Corr (X, \A) \; .
\end{equation}
The free abelian group $\Z \Oh (X)$ on the set $\Oh (X)$ carries two ring structures since we may view it as the monoid algebra of either $(\Oh(X) , \cdot)$ or $(\Oh (X) , +)$. The corresponding multiplication maps on $\Corr (X, \A)$ are given by composition 
\[
\Corr (X, \A) \times \Corr (X, \A) \xrightarrow{c} \Corr (X , \A \times \A) \xrightarrow{m_* , a_*} \Corr (X , \A) \; .
\]
Here $c ( \alpha , \beta) = p^*_1 (\alpha) \cdot p^*_2 (\beta)$ with the projections $p_1$ and $p_2$ in 
\[
X \times \A \times \A = (X \times \A) \times_X (X \times \A) \overset{p_1}{\underset{p_2}\rightrightarrows} X \times \A \; .
\]
The maps $m_*$ and $a_*$ are induced by the multiplication and addition maps of the ring scheme $\A$ (or $\A_X$). Let $O \in \Corr (X , \A)$ be the image of $0 \in \Oh (X)$ under the map \eqref{eq:23}. Then $\Z \cdot O$ is the ideal generated by $O$ in the ring $(\Corr (X , \A) , + , m_* \verk c)$. Let $\uCorr (X, \A)$ be the quotient ring by this ideal. From the map \eqref{eq:25} we obtain a functorial ring homomorphism from the reduced monoid ring of $(\Oh (X) , \cdot)$
\begin{equation}
\label{eq:25}
\uZ \Oh (X) \longrightarrow \uCorr (X , \A) \; .
\end{equation}

\begin{theorem}
\label{t5.1}
On the category of normal, Noetherian, affine schemes $X$ there is a unique functorial factorization of the map \eqref{eq:25} over a functorial ring isomorphism
\begin{equation}
\label{eq:26}
W_{\rat} (\Oh (X)) \silo \uCorr (X , \A) \; .
\end{equation}
\end{theorem}

\begin{proof}
By \eqref{eq:22} we have identifications
\[
\Corr (X, \A) := c_{\equi} (\A^1_X / X , 0) = \PropCycl_{\equi} (\A^1_X / X) = \PropCaDir (\A^1_X / X) \; .
\]
Thus we may identify $\Corr (X , \A)$ with the free abelian group generated by integral closed subschemes $Z$ of $\A^1_X = \spec A [T]$ which are finite and flat over $X = \spec A$. Via $Z = \spec A [T] / \ep$ these subschemes correspond to the prime ideals $\ep$ of $A [T]$ for which $B = A [T] / \ep$ is a finite faithfully flat $A$-algebra. Since $b$ is integral over $A$, the minimal polynomial $P (T) \in K [T]$ over $K$ of $b = T \mod \ep \in B$ has coefficients in the integral closure of $A$ in $K$ and hence in $A$ since $A$ is normal. Thus we have $B = A [T] / (P (T))$ for a monic polynomial $P (T) \in A [T]$ which is irreducible in $K [T]$. Since a normal Noetherian domain is completely integrally closed by Proposition \ref{t1.6} it follows from Theorem \ref{t1.5}, c) that $W_{\rat} (A) = W_{\rat} (K) \cap W (A)$ is the free abelian group generated by the polynomials $0 \neq f (T) \in A [T]$ with $f (0) = 1$ which are irreducible in $K [T]$. Sending such a polynomial $f$ to $Z = \spec (A (T) / (P (T)))$ where $P (T) = f^* (T) := T^{\deg f} f (T^{-1})$ we obtain an injective homomorphism of additive groups from $W_{\rat} (A) = \Z [f's]$ to $\Corr (X , \A)$ which becomes an isomorphism after projecting to $\uCorr (X , \A)$. Intuitively, $f$ is sent to the multivalued map on $X$ whose values are the inverses of the zeroes of $f$. It follows from the definition that the map \eqref{eq:25} factors over the map \eqref{eq:26}. After checking functoriality, the uniqueness assertion and the compatibilites with the multiplicative structures follow by using \cite[Lemma 10.136.14]{stacks}.
\end{proof}

Consider the finite, flat $N$-th power map $\pi_N = (\;)^N : \A \to \A$ of $\A$ viewed as a ring scheme. It induces endomorphisms $\pi_{N*}$ and $\pi^*_N$ of $\Corr (X, \A)$ and of $\uCorr (X, \A)$ as rings, resp. abelian groups.

\begin{prop}
\label{t5.2}
Under the isomorphism $W_{\rat} (\Oh (X)) \silo \uCorr (X, A)$ of Theorem \ref{t5.1}, Frobenius $F_N$ and Verschiebung $V_N$ on $W_{\rat} (\Oh (X))$ correspond to $\pi_{N*}$ and $\pi^*_N$ on $\uCorr (X , \A)$.
\end{prop}

\begin{proof}
For $f \in A = \Oh (X)$ with $T^N- f$ a product of linear factors in $A$ one checks $F_N (f) \ent \pi_{N*} (\Gamma_f)$ and $V_N (f) \ent \pi^*_N (\Gamma_f)$ by applying the definitions. The general case is reduced to this one by using \cite[Lemma 10.136.14]{stacks}.
\end{proof}

\begin{rems}
1) It seems to be an interesting question whether the isomorphism \eqref{eq:26} can be extended to non-normal schemes $X$.\\
2) Let $\End_A$ denote the exact category of finite rank, projective $A$-modules $M$ equipped with an endomorphism $\varphi$. The $K$-group $K_0 (\End_A)$ whose addition is induced by the direct sum becomes a commutative unital ring with multiplication induced by tensor product. Sending $(M , \varphi)$ to $\det (1 - \varphi T \,|\, M)$ induces a ring homomorphism
\begin{equation}
\label{eq:27}
K_0 (\End_A) \longrightarrow W_{\rat} (A) \ .
\end{equation}
Via the map induced by $M \mapsto (M , 0)$, the ring $K_0 (\End_A)$ contains $K_0 (A)$ as an ideal and via the idempotent $e = [(A,0)]$ as a direct factor. Restricted to $\tK_0 (\End_A) = K_0 (\End_A) / K_0 (A) \equiv (1 - e) K_0 (\End_A)$ the map \eqref{eq:27} becomes an isomorphism as Almkvist showed in \cite{A}
\begin{equation}
\label{eq:28}
\tK_0 (\End_A) \silo W_{\rat} (A) \; .
\end{equation}
In the light of Theorem \ref{t5.1} the map \eqref{eq:27} can be reformulated. For an affine morphism $\pi : Y \to X$ of Noetherian schemes, let $\Ch_{Y , X}$ be the category of quasicoherent sheaves $\Mh$ on $Y$ for which $\pi_* \Mh$ is a vector bundle. The category $\Ch_{Y /X}$ is exact since $\pi_*$ is exact. Reducing to the affine case, one shows that for $\Mh$ in $\Ch_{Y / X}$ the scheme theoretic support $Z = \supp \Mh$ is finite over $X$. Let $\{ Z_i \}$ be the irreducible components of $Z$ and $\eta_i \in Z_i$ their generic points. We set
\[
\cycle (\Mh) = \sum_i (\length_{\Oh_{Y , \eta_i}} \Mh_{\eta_i}) [Z_i] \in \PropCycl (Y / X , 0) \; .
\]
Now consider the special case where $X = \spec A$ is as in Theorem \ref{t5.1} and $Y = \A^1_X$. Then an object $\Mh$ of $\Ch_{Y / X}$ corresponds to an $A [T]$-module $M$ which is projective of finite rank as an $A$-module. Thus the categories $\End_A$ and $\Ch_{Y / X}$ are equivalent and hence
\[
K_0 (\Ch_{Y / X}) = K_0 (\End_A) \; .
\]
By straightforward arguments, one checks that the following diagram commutes
\begin{equation}
\label{eq:29}
\vcenter{\xymatrix{
K_0 (\Ch_{Y / X}) \ar[r]^-{\cycle} \ar@{=}[d] & \PropCycl_{\equi} (Y / X , 0) \ar@{=}[r] & \Corr (X, \A) \ar[d] \\
K_0 (\End_A) \ar[r] & W_{\rat} (A) \ar[r]^-{\overset{\eqref{eq:26}}{\sim}} & \uCorr (X, \A) \; .
}}
\end{equation}
We hope to return to the question whether Almkvist's theorem can be generalized to a comparison result between $K_0 (\Ch_{Y, X})$ and $c_{\equi} (Y/X , 0)$. Considering suitable categories of perfect complexes one would also want to get rid of the condition that $\pi$ is affine. Moreover the comparison should be extended to higher $K$-groups and higher cycle groups.
\end{rems}


\begin{thebibliography}{DPJM22}

\bibitem[Alm74]{A}
Gert Almkvist.
\newblock The {G}rothendieck ring of the category of endomorphisms.
\newblock {\em J. Algebra}, 28:375--388, 1974.

\bibitem[BB18]{BB}
Andrew Bakan and Christian Berg.
\newblock Solvability of the {H}ankel determinant problem for real sequences.
\newblock In {\em Frontiers in orthogonal polynomials and {$q$}-series},
  volume~1 of {\em Contemp. Math. Appl. Monogr. Expo. Lect. Notes}, pages
  85--117. World Sci. Publ., Hackensack, NJ, 2018.

\bibitem[Ben70]{B}
Benali Benzaghou.
\newblock Alg\`ebres de {H}adamard.
\newblock {\em Bull. Soc. Math. France}, 98:209--252, 1970.

\bibitem[Bou64]{Bou}
N.~Bourbaki.
\newblock {\em \'El\'ements de math\'ematique. {F}asc. {XXX}. {A}lg\`ebre
  commutative. {C}hapitre 5: {E}ntiers. {C}hapitre 6: {V}aluations}, volume No.
  1308 of {\em Actualit\'es Scientifiques et Industrielles [Current Scientific
  and Industrial Topics]}.
\newblock Hermann, Paris, 1964.

\bibitem[BV93]{BV}
J.~Bryden and K.~Varadarajan.
\newblock Witt vectors which are rational functions.
\newblock {\em Comm. Algebra}, 21(11):4263--4270, 1993.

\bibitem[CC75]{CC}
Paul-Jean Cahen and Jean-Luc Chabert.
\newblock \'el\'ements quasi-entiers et extensions de {F}atou.
\newblock {\em J. Algebra}, 36(2):185--192, 1975.

\bibitem[Cha72]{C1}
Jean-Luc Chabert.
\newblock Anneaux de {F}atou.
\newblock {\em Enseign. Math. (2)}, 18:141--144, 1972.

\bibitem[Den24]{D}
Christopher Deninger.
\newblock Dynamical systems for arithmetic schemes, 2024.
\newblock https://arxiv.org/abs/1807.06400; {T}o appear in {I}ndagationes
  {M}ath.

\bibitem[DPJM22]{DMP}
Emanuele Dotto, Irakli Patchkoria, and Kristian Jonsson~Moi.
\newblock Witt vectors, polynomial maps, and real topological {H}ochschild
  homology.
\newblock {\em Ann. Sci. \'Ec. Norm. Sup\'er. (4)}, 55(2):473--535, 2022.

\bibitem[GK15]{GK}
Ofer Gabber and Shane Kelly.
\newblock Points in algebraic geometry.
\newblock {\em J. Pure Appl. Algebra}, 219(10):4667--4680, 2015.

\bibitem[Gra77]{Gr}
Daniel~R. Grayson.
\newblock The {$K$}-theory of endomorphisms.
\newblock {\em J. Algebra}, 48(2):439--446, 1977.

\bibitem[Haz83]{H}
Michiel Hazewinkel.
\newblock Operations in the {$K$}-theory of endomorphisms.
\newblock {\em J. Algebra}, 84(2):285--304, 1983.

\bibitem[Kos]{Ko}
Aleksey Kostenko.
\newblock Hankel operators and applications (lecture notes).
\newblock https://users.fmf.uni-lj.si/kostenko/teach/HankelNotes.pdf.

\bibitem[Kru36]{K}
Wolfgang Krull.
\newblock Beitr\"age zur {A}rithmetik kommutativer {I}ntegrit\"atsbereiche.
\newblock {\em Math. Z.}, 41(1):665--679, 1936.

\bibitem[KS16]{KS}
Robert~A. Kucharczyk and Peter Scholze.
\newblock Topological realisations of absolute {G}alois groups, 2016.
\newblock arXiv:1609.04717.

\bibitem[Lut25]{L}
Judith Lutz.
\newblock $p$-adic points of rational {W}itt schemes., 2025.
\newblock {PhD}-thesis, {M}uenster.

\bibitem[MVW06]{MVW}
Carlo Mazza, Vladimir Voevodsky, and Charles Weibel.
\newblock {\em Lecture notes on motivic cohomology}, volume~2 of {\em Clay
  Mathematics Monographs}.
\newblock American Mathematical Society, Providence, RI; Clay Mathematics
  Institute, Cambridge, MA, 2006.

\bibitem[Ric]{R}
Timo Richarz.
\newblock Basics on affine {G}rassmannians.
\newblock https://www.mathematik.tu-darmstadt.de/media/algebra/homepages/
  richarz/Notes\_on\_affine\_Grassmannians.pdf.

\bibitem[Sam64]{S}
P.~Samuel.
\newblock {\em Lectures on unique factorization domains}, volume No. 30 of {\em
  Tata Institute of Fundamental Research Lectures on Mathematics}.
\newblock Tata Institute of Fundamental Research, Bombay, 1964.
\newblock Notes by M. Pavman Murthy.

\bibitem[Sch01]{Sch}
Hans Schoutens.
\newblock Recursive sequences and faithfully flat extensions.
\newblock {\em Rocky Mountain J. Math.}, 31(4):1423--1427, 2001.

\bibitem[{Sta}]{stacks}
{Stacks Project Authors}.
\newblock {Stacks Project}.
\newblock Available at {\tt https://stacks.math.columbia.edu}.

\bibitem[Sti82]{St1}
Jan Stienstra.
\newblock Operations in the higher {$K$}-theory of endomorphisms.
\newblock In {\em Current trends in algebraic topology, {P}art 1 ({L}ondon,
  {O}nt., 1981)}, volume~2 of {\em CMS Conf. Proc.}, pages 59--115. Amer. Math.
  Soc., Providence, RI, 1982.

\bibitem[Sti85]{St2}
Jan Stienstra.
\newblock Cartier-{D}ieudonn\'e{} theory for {C}how groups.
\newblock {\em J. Reine Angew. Math.}, 355:1--66, 1985.
\newblock Correction to: ``Cartier-{D}ieudonn\'e{} theory for {C}how groups'',
  {J}. {R}eine {A}ngew. {M}ath. 362 (1985), 218–220.

\bibitem[SV96]{SV1}
Andrei Suslin and Vladimir Voevodsky.
\newblock Singular homology of abstract algebraic varieties.
\newblock {\em Invent. Math.}, 123(1):61--94, 1996.

\bibitem[SV00]{SV2}
Andrei Suslin and Vladimir Voevodsky.
\newblock Relative cycles and {C}how sheaves.
\newblock In {\em Cycles, transfers, and motivic homology theories}, volume 143
  of {\em Ann. of Math. Stud.}, pages 10--86. Princeton Univ. Press, Princeton,
  NJ, 2000.

\bibitem[Voe96]{V}
V.~Voevodsky.
\newblock Homology of schemes.
\newblock {\em Selecta Math. (N.S.)}, 2(1):111--153, 1996.

\end{thebibliography}
\end{document}